\documentclass[11pt,reqno]{article}
\usepackage{latexsym, amsmath, amssymb, amsthm, a4, epsfig}
\usepackage{graphicx}
\usepackage{xcolor}

\newtheorem{theorem}{Theorem}[section]
\newtheorem{cor}[theorem]{Corollary}
\newtheorem{lemma}[theorem]{Lemma}
\newtheorem{prop}[theorem]{Proposition}

\newcommand{\nm}{\noalign{\smallskip}}
\newcommand{\ds}{\displaystyle}
\newcommand{\p}{\partial}

\newcommand{\eqnref}[1]{(\ref {#1})}

\newcommand{\Cbb}{\mathbb{C}}

\newcommand{\Kbb}{\mathbb{K}}
\newcommand{\Rbb}{\mathbb{R}}
\newcommand{\Sbb}{\mathbb{S}}
\newcommand{\Zbb}{\mathbb{Z}}

\newcommand{\la}{\langle}
\newcommand{\ra}{\rangle}

\newcommand{\Hcal}{\mathcal{H}}

\newcommand{\Kcal}{\mathcal{K}}
\newcommand{\Dcal}{\mathcal{D}}

\newcommand{\Scal}{\mathcal{S}}

\newcommand{\Ga}{\alpha}
\newcommand{\Gb}{\beta}

\newcommand{\Ge}{\epsilon}

\newcommand{\Gvf}{\varphi}

\newcommand{\Gl}{\lambda}
\newcommand{\Gn}{\eta}
\newcommand{\Gm}{\mu}
\newcommand{\Gv}{\nu}

\newcommand{\Gt}{\theta}

\newcommand{\Gr}{\rho}

\newcommand{\Gs}{\sigma}

\newcommand{\Gy}{\psi}
\newcommand{\Gz}{\zeta}
\newcommand{\GD}{\Delta}

\newcommand{\GG}{\Gamma}
\newcommand{\GL}{\Lambda}

\newcommand{\GT}{\Theta}

\newcommand{\GO}{\Omega}

\newcommand{\beq}{\begin{equation}}
\newcommand{\eeq}{\end{equation}}

\def\ol{\overline}

\newcommand{\zbar}{\bar{z}}
\newcommand{\Gzbar}{\bar{\Gz}}

\numberwithin{equation}{section}
\numberwithin{figure}{section}

\begin{document}

\title{Spectrum of the Neumann-Poincar\'e operator and optimal estimates for transmission problems in presence of two circular inclusions\thanks{\footnotesize This work was supported by NRF (of S. Korea) grants No. 2019R1A2B5B01069967.}}

\author{Yong-Gwan Ji\thanks{Department of Mathematics and Institute of Applied Mathematics, Inha University, Incheon 22212, S. Korea (22151063@inha.edu, hbkang@inha.ac.kr).} \and Hyeonbae Kang\footnotemark[2] }

\maketitle

\begin{abstract}
We consider the field concentration for the transmission problems of the homogeneous and inhomogeneous conductivity equations in the presence of closely located circular inclusions. We revisit these well-studied problems by exploiting the spectral nature residing behind the phenomenon of the field concentration. The spectral approach enables us not only to recover the existing results with new insights but also to produce significant new results. We show that when relative conductivities of inclusions have different signs, then the gradient of the solution is bounded regardless of the distance between inclusions, but the second and higher derivatives may blow up as the distance tends to zero if one of conductivities is $0$ and the other $\infty$. This result holds for both homogenous and inhomogeneous problems. We prove these results by precise quantitative estimates of the derivatives of the solution. We also show by examples that the estimates are optimal.
\end{abstract}



\section{Introduction}

During last three decades or so, there has been significant progress in quantitative analysis of the field concentration (field enhancement or stress) in the narrow region between two inclusions for the homogeneous and inhomogeneous transmission problems of conductivity equations (see, for example, \cite{AKLLL, AKL, BLY, DL, LV, Yun}). In this paper we revisit the problem by exploiting the spectral nature residing behind the phenomenon of the field concentration when inclusions are circular. Through the spectral approach we not only gain new insights on the existing results but also obtain significant new results. We show that when $(k_1-1)(k_2-1) <0$, where $k_1, \ k_2$ are conductivities of inclusions and $1$ is conductivity of the background, have different signs (typically, $k_1=0$ and $k_2=\infty$), then the gradient of the solution is bounded regardless of the distance between inclusions, but the second and higher derivatives may blow up as the distance tends to zero. We show that this result holds for both homogeneous and inhomogeneous problems.

We now state the main results of this paper in a precise manner and then compare them with existing results. Let $D_1$ and $D_2$ be disks in $\Rbb^2$ which are separated by the small distance denoted by $\Ge$. Let $D:=D_1 \cup D_2$. Let $k_j$ be the conductivity of $D_j$ for $j=1,2$, while that of $\Rbb^2 \setminus D$ is assumed to be $1$. So the conductivity distribution is given by
\beq
\Gs = \chi_{\Rbb^2 \setminus D} + k_1 \chi_{D_1}  + k_2 \chi_{D_2},
\eeq
where $\chi$ denotes the characteristic function on the respective set. Assuming that $0<k_j \neq 1 <\infty$ ($j=1,2$), we consider the inhomogeneous transmission problem: for a given function $f$
\beq\label{NCE}
\begin{cases}
	\nabla \cdot \Gs \nabla u = f \quad &\text{in } \Rbb^2, \\
	u(x) = c \ln |x| + O(|x|^{-1}) \quad &\text{as } |x| \to \infty
\end{cases}
\eeq
for some constant $c$. We also consider the homogeneous transmission problem:
\beq\label{HCE}
\begin{cases}
\nabla \cdot \Gs \nabla u = 0 \quad &\text{in } \Rbb^2, \\
u(x) - H(x) = O(|x|^{-1}) \quad &\text{as } |x| \to \infty,
\end{cases}
\eeq
where $H$ is a given function harmonic in $\Rbb^2$.

Note that there are three parameters involved in the problem \eqnref{NCE}: conductivities $k_1$, $k_2$, and the distance $\Ge$ between two inclusions. The problem is to derive estimates for derivatives of $u$ in terms of these parameters when $\Ge$ tends to $0$.

We assume that the inhomogeneity $f$ is given by $f=\nabla \cdot g$ for some $g$. It is assumed that $g$ is compactly supported in $\Rbb^2$, which is just for the sake of simplicity and can be replaced with some integrability condition. We also impose some regularity condition on $g$: $g$ is piecewise $C^{n, \Ga}$ for some non-negative integer $n$ and $0<\Ga<1$. The function $g$ being piecewise $C^{n, \Ga}$ in this paper means that $g$ is $C^{n,\Ga}$ on $\ol{D_1}$, $\ol{D_2}$ and $\Rbb^2 \setminus D$ separately. For piecewise $C^{n, \Ga}$ functions $g$, the norm is defined by
\beq\label{Cknorm}
\| g \|_{n,\Ga}:= \|g\|_{C^{n,\Ga}(\ol{D_1})} +  \|g\|_{C^{n,\Ga}(\ol{D_2})}+ \|g\|_{C^{n,\Ga}(\Rbb^2 \setminus D)}.
\eeq
When $\Ga=0$, we denote it by $\| g \|_{n,0}$. We also use the following norm:
\beq\label{Ckstarnorm}
\| g \|_{n,\Ga}^*:= \frac{1}{k_1} \|g\|_{C^{n,\Ga}(\ol{D_1})} +  \frac{1}{k_2} \|g\|_{C^{n,\Ga}(\ol{D_2})}+ \|g\|_{C^{n,\Ga}(\Rbb^2 \setminus D)}.
\eeq
We also denote by $\| u \|_{n,\GO}$ the piece-wise $C^n$ norm on $\GO$ when $\GO$ is a bounded set containing $\ol{D_1 \cup D_2}$.

When $(k_1-1)(k_2-1) >0$, we obtain the following theorems for the inhomogeneous and homogeneous transmission problems. Here and throughout this paper, we put
\beq\label{notation2}
r_*:= \sqrt{\frac{2(r_1+r_2)}{r_1 r_2}}
\eeq
for ease of notation ($r_j$ is the radius of $D_j$, $j=1,2$). We also put
\beq\label{Glj}
\Gl_j= \frac{k_j+1}{2(k_j-1)}, \quad j=1,2.
\eeq

\begin{theorem}\label{thm:main1}
Suppose $(k_1-1)(k_2-1) >0$ and $f=\nabla \cdot g$ for some piecewise $C^{n-1, \Ga}$ function $g$ with the compact support ($n$ is a positive integer and $0<\Ga<1$). There is a constant $C>0$ independent of $k_1$, $k_2$, $\Ge$ and $g$ such that the solution $u$ to \eqnref{NCE} satisfies
\beq\label{mainest1}
\| u \|_{n,0} \le C \| g \|_{n-1,\Ga}^* \left( 4\Gl_1\Gl_2-1 + r_* \sqrt{\Ge} \right)^{-n}.
\eeq
This estimate is optimal in the sense that there is $f$ such that the reverse inequality (with a different constant $C$) holds when $n=1$.
\end{theorem}

\begin{theorem}\label{thm:main4}
Let $\GO$ be a bounded set containing $\ol{D_1 \cup D_2}$. Let $u$ be the solution to \eqnref{HCE}. If $(k_1-1)(k_2-1) >0$, then there is a constant $C>0$ independent of $k_1$, $k_2$, $\Ge$ and the function $H$ such that
\beq\label{mainest4-1}
\| u \|_{n,\GO} \le C \| H \|_{C^n (\GO)} \left( 4\Gl_1\Gl_2-1 + r_* \sqrt{\Ge} \right)^{-n} .
\eeq
This estimate is optimal in the sense that there is a harmonic function $H$ such that the reverse inequality (with a different constant $C$) holds for the case $n=1$.
\end{theorem}

Since
$$
4\Gl_1\Gl_2-1 = \frac{2(k_1+k_2)}{(k_1-1)(k_2-1)},
$$
the estimate \eqnref{mainest1} shows that if either $k_1$ or $k_2$ is finite (away from $0$ and $\infty$), then $\| u \|_{n,0}$ is bounded regardless of the distance $\Ge$, while if both $k_1$ and $k_2$ tend to $\infty$, then the right-hand side of \eqnref{mainest1} is of order ${\Ge}^{-n/2}$. As shown by an example in section \ref{sec:opti}, $\nabla u$ may actually blow up at the order of ${\Ge}^{-1/2}$ . If $k_1$ and $k_2$ tend to $0$, then the right-hand side of \eqnref{mainest1} is also of order ${\Ge}^{-n/2}$ provided that $\| g \|_{n-1,\Ga}^*$ is bounded, in particular, if there is no source in $D_1 \cup D_2$, namely, $g=0$ in $D_1 \cup D_2$. The estimate \eqnref{mainest4-1} yields the same findings.

The estimates \eqnref{mainest1} and \eqnref{mainest4-1} are not new; Here, we derive them using a different method: the spectrum of the Neumann-Poincar\'e operator related to the problems. The estimate \eqnref{mainest1} was obtained in \cite{DL}.
The estimate \eqnref{mainest4-1} has a long history: The estimate for the gradient, namely, for $n=1$, is obtained in \cite{AKLLL, AKL}; that for higher $n$ in \cite{DL}. The results for circular inclusions have been extended to those of more general shape with the conductivity $\infty$ \cite{BLY, KY19, Lekner-PRSA-12, LY-CPDE-09, Yun}, the insulating problem in three dimensions or higher \cite{BLY10, LY, Wein, Yun16}, and to other equations; $p-$Laplacian \cite{GN-MMS-12}, Lam\'e system \cite{BLL-ARMA-15, BLL-AM-17, KY}, and Stokes system \cite{AKKY}. The estimate \eqnref{mainest4-1} says in particular that $\nabla u$ is bounded regardless of the distance between two inclusions if $k_1$ or $k_2$ is finite. This fact is known to be true in a more general setting \cite{BV, LN, LV}. Asymptotic characterizations of the gradient blow-up when $k_1=k_2=\infty$ (or $k_1=k_2=0$) are obtained in \cite{ACKLY13, KLY, KLY15, LY-JMAA-15}.

Theorem \ref{thm:main1} and \ref{thm:main4} (and theorems to follow) are proved by exploiting the spectral nature of the problem \eqnref{NCE}. The problem \eqnref{NCE} can be reduced to the following problem (see the beginning of Section \ref{sec:NP}):
\beq\label{mainproblem}
\begin{cases}
\GD v = 0  \quad&\mbox{in } D \cup (\Rbb^2 \setminus \ol{D}), \\
\ds v|_+ - v|_- = 0  \quad&\mbox{on }\p D,  \\
\ds \p_\nu v|_+ - k_j \p_\nu v|_- = (k_j-1)\Gn_j  \quad&\mbox{on }\p D_j, \ j=1,2, \\
\ds v (x) = c \ln |x| + O( |x|^{-1})~&\mbox{as }|x|\rightarrow\infty
\end{cases}
\eeq
for some constant $c$. Here and throughout this paper, $\p_\nu$ denotes the outward normal derivative on $\p D_j$ and the subscripts $\pm$ denote the limits from outside and inside of $D_j$ respectively. . For the problem \eqnref{NCE}, $\Gn_j$ appearing in \eqnref{mainproblem} is given by $\Gn_j=\p_\nu F|_{\p D_j}$ ($j=1,2$), where $F$ is the (weighted) Newtonian potential of $f$, namely,
\beq\label{Fdef}
F(x)=  \frac{1}{2\pi} \int_{\Rbb^2 \setminus D} \ln|x-y| f(y) \, dy + \sum_{j=1}^2 \frac{1}{2 \pi k_j} \int_{D_j} \ln|x-y| f(y) \, dy
\eeq
for $x \in \Rbb^2$. For the problem \eqnref{HCE}, $\Gn_j=\p_\nu H|_{\p D_j}$.

If we represent the solution to \eqnref{mainproblem} in terms of the single layer potential, then the interface condition on $\p D_j$ for $j=1,2$ (the third line in \eqnref{mainproblem}) is reduced to an integral equation for the Neumann-Poincar\'e (abbreviated by NP) operator on two circles $\p D_1$ and $\p D_2$ (see section \ref{sec:NP}). By solving the integral equation we see explicitly that the gradient and higher derivatives of the solution to \eqnref{mainproblem} is quantitatively estimated in terms of the quantity $4 \Gl_1 \Gl_2-\Gr$, where $\Gr$ is a constant determined by the radii $r_1$ and $r_2$ of inclusions $D_1$ and $D_2$ and the distance $\Ge$ between them such that $\pm \frac{1}{2} \Gr^n$ ($n=1, 2, \ldots$) are eigenvalues of the NP operator (see section \ref{sec:twocircles} and \ref{sec:specrep}). Since $\Gr \sim 1- r_* \sqrt{\Ge}$ (see \eqnref{Grasymp}), we have estimates \eqnref{mainest1} and \eqnref{mainest4-1}. Here and throughout this paper $A \sim B$ for some positive quantities $A$ and $B$ depending on some parameters means that there are positive constants $C_1$ and $C_2$ independent of those parameters such that $C_1 \le A/B \le C_2$. Notation $\lesssim$ and $\gtrsim$ are defined analogously.

If $(k_1-1)(k_2-1) <0$, then $4 \Gl_1 \Gl_2 <0$. Thus the right-hand sides of \eqnref{mainest1} and \eqnref{mainest4-1} are bounded and are not optimal bounds in this case. We obtain the following theorem for this case.

\begin{theorem}\label{thm:main3}
Suppose $(k_1-1)(k_2-1) <0$ and $f=\nabla \cdot g$ for some piecewise $C^{n, \Ga}$ function $g$ with the compact support ($n$ is a positive integer and $0<\Ga<1$). There is a constant $C>0$ independent of $k_1$, $k_2$, $\Ge$ and $g$ such that the solution $u$ to \eqnref{NCE} satisfies
\beq\label{mainest3-1}
\| u \|_{n,0} \le C \| g \|_{n,\Ga}^* \left( 4|\Gl_1 \Gl_2|-1 + r_* \sqrt{\Ge} \right)^{-n+1}.
\eeq
This estimate is optimal in the sense that there is $f$ such that the reverse inequality (with a different constant $C$) holds for $n=2$.
\end{theorem}

The estimate \eqnref{mainest3-1} shows that if $(k_1-1)(k_2-1) <0$, then $\nabla u$ is bounded regardless of the $k_1$, $k_2$ and $\Ge$, provided that $\| g \|_{n,\Ga}^*$ is bounded. But, the $n$th ($n \ge 2$) order derivative may blow up at the rate of $\Ge^{-(n-1)/2}$ if, for example, $k_1=0$ and $k_2=\infty$. The second derivative of $u$ actually blows up at the rate of $\Ge^{-1/2}$ in some case as we show in section \ref{sec:opti}.

We obtain the following theorem for the homogeneous problem similar to Theorem \ref{thm:main3}.

\begin{theorem}\label{thm:main5}
Let $\GO$ be a bounded set containing $\ol{D_1 \cup D_2}$. Let $u$ be the solution to \eqnref{HCE}. If $(k_1-1)(k_2-1) <0$, then there is a constant $C>0$ independent of $k_1$, $k_2$, $\Ge$ and the function $H$ such that
\beq\label{mainest4-2}
\| u \|_{n,\GO} \le C \| H \|_{C^{n+1} (\GO)} \left( 4|\Gl_1 \Gl_2|-1 + r_* \sqrt{\Ge} \right)^{-n+1}.
\eeq
This estimate is optimal in the sense that there is a harmonic function $H$ such that the reverse inequality (with a different constant $C$) holds for $n=2$.
\end{theorem}

This paper is organized as follows. In section \ref{sec:NP}, we review the spectral theory of the NP operator defined on the boundary of two disjoint inclusions of general shape and prove uniqueness and existence of the solution to \eqnref{mainproblem}. Section \ref{sec:twocircles} is to show M\"obius invariance of the NP spectra and derive the full spectrum on two circles. In section \ref{sec:specrep} we derive representation formulas for the solutions to \eqnref{NCE} and \eqnref{HCE} in terms of NP spectrum. Section \ref{sec:pfs} is to prove Theorem \ref{thm:main1}-\ref{thm:main5} and section \ref{sec:opti} is to show optimality of estimates in those theorems.

\section{The NP operator and well-posedness of the problem}\label{sec:NP}

Let us begin this section by showing how the problem \eqnref{NCE} is reduced to \eqnref{mainproblem}. The problem \eqnref{NCE} can be expressed as
\beq\label{200}
\begin{cases}
\GD u = f  \quad&\mbox{in } \Rbb^2 \setminus \ol{D}, \\
\GD u = k_j^{-1} f  \quad&\mbox{in } D_j, \ j=1,2 , \\
\ds u|_+ - u|_- = 0  \quad&\mbox{on }\p D \\
\ds \p_\nu u|_+ - k_j \p_\nu u|_- = 0  \quad&\mbox{on }\p D_j, \ j=1,2 ,\\
\ds u (x) = c \ln|x| + O( |x|^{-1})~&\mbox{as }|x|\rightarrow\infty.
\end{cases}
\eeq
Let $v:=u-F$ with $F$ given by \eqnref{Fdef}. Since $\GD F= f$ in $\Rbb^2 \setminus \ol{D}$ and $\GD F = k_j^{-1} f$ in $D_j$, $v$ is the solution to \eqnref{mainproblem}. By putting $v:u-H$, we see that \eqnref{HCE} is reduced to \eqnref{mainproblem}.

In this section we show that the problem \eqnref{mainproblem} is reduced to an integral equation of the NP operator. We also prove uniqueness and existence of the solution to \eqnref{mainproblem} for $\Gn_j \in H^{-1/2} (\p D_j)$ for $j=1,2$, and of \eqnref{NCE} as a consequence. Here and afterwards, $H^{-1/2} (\p D_j)$ denotes the Sobolev space of order $-1/2$ on $\p D_j$, and $H^{-1/2}_0 (\p D_j)$ is the subspace of $H^{-1/2} (\p D_j)$ whose element $\Gn_j$ satisfies
\beq
\int_{\p D_j} \Gn_j d\Gs =0.
\eeq
We use $\Hcal^{-1/2}=\Hcal^{-1/2}(\p D)$ and $\Hcal^{-1/2}_0=\Hcal_0^{-1/2}(\p D)$ respectively to denote $H^{-1/2} (\p D_1) \times H^{-1/2} (\p D_2)$ and $H^{-1/2}_0 (\p D_1) \times H^{-1/2}_0 (\p D_2)$ for ease of notation.

\subsection{Neumann-Poincar\'e operator}

Let $\GO$ be a simply connected bounded domain in $\Rbb^2$ with the Lipschitz continuous boundary. For a function $\Gm \in H^{-1/2}(\p \GO)$, the single layer potential $\Scal_{\p \GO} [\Gm]$ is defined by
\beq
\Scal_{\p \GO} [\Gm] (x):= \frac{1}{2\pi} \int_{\p \GO} \ln |x-y| \; \Gm (y) \, d\Gs(y), \quad x \in \Rbb^2.
\eeq
The following formulas for the single layer potential are well known (see, for example, \cite{AmKa07Book2, Fo95}) :
\begin{align}
\Scal_{\p \GO} [\Gm] \big|_+ (x) &= \Scal_{\p \GO} [\Gm] \big|_- (x) \quad \text{ on } \p \GO, \label{singlejump1} \\
\p_\nu \Scal_{\p \GO} [\Gm] \big|_\pm (x) &= \biggl( \pm \frac{1}{2} I + \Kcal_{\p \GO}^* \biggr) [\Gm] (x) \quad \text{ on } \p \GO. \label{singlejump2}
\end{align}
Here the operator $\Kcal_{\p \GO}^*$ is defined by
\beq\label{introkd}
\Kcal_{\p \GO}^* [\Gm] (x):=
\frac{1}{2\pi}\int_{\p \GO} \frac{(x-y) \cdot \nu_x}{|x-y|^2} \; \Gm(y)\,d\Gs(y), \quad x \in \p  \GO,
\eeq
and is called the NP operator on $\p \GO$.

The NP operator $\Kcal_{\p \GO}^*$ is not self-adjoint on $L^2(\p\GO)$ unless $\p \GO$ is a circle. However it can be realized as a self-adjoint operator on $H^{-1/2}(\p\GO)$ using the Plemelj's symmetrization principle:
\beq
\Scal_{\p \GO} \Kcal_{\p \GO}^* = \Kcal_{\p \GO} \Scal_{\p \GO}
\eeq
(see \cite{KPS}). The NP operator $\Kcal^*_{\p \GO}$ is compact if $\p \GO $ is $C^{1,\Ga}$, and hence its eigenvalues are real, of finite multiplicities, and accumulate to 0.

We now consider the single layer potential on $\p D= \p D_1 \cup \p D_2$. For this section we assume that $D_1$ and $D_2$ are bounded domains with Lipschitz continuous boundaries for some $\Ga >0$. For $\Gvf := (\Gvf_1, \Gvf_2)^T \in \Hcal^{-1/2}(\p D)$, let
$$
\Scal_{\p D}[\Gvf](x) := \Scal_{\p D_1}[\Gvf_1](x) +  \Scal_{\p D_2}[\Gvf_2](x), \quad x \in \Rbb^2.
$$
We seek a solution to \eqnref{mainproblem} (assuming that $\Gn =(\Gn_1,\Gn_2)^T \in \Hcal_0^{-1/2}(\p D)$) in the form of
\beq\label{inteqn1}
u(x) := \Scal_{\p D}[\Gvf](x)
\eeq
for some $\Gvf \in \Hcal_0^{-1/2}(\p D)$.

It follows from the jump relation \eqnref{singlejump2} and the jump condition of the flux (the third line in \eqnref{mainproblem}) that $\Gvf$ should satisfy the following integral equations on $\p D_1$ and $\p D_2$, respectively:
$$
(1/2I+\Kcal_{\p D_1}^*)[\Gvf_1] + \p_{\nu} \Scal_{\p D_2}[\Gvf_2]
- k_1 \big( (-1/2I+\Kcal_{\p D_1}^*)[\Gvf_1] + \p_\nu \Scal_{\p D_2}[\Gvf_2] \big) = (k_1-1)\Gn_1,
$$
and
$$
\p_\nu \Scal_{\p D_1}[\Gvf_1] + (1/2I+ \Kcal_{\p D_2}^*)[\Gvf_2] - k_2 \big(
\p_{\nu} \Scal_{\p D_1}[\Gvf_1] + (-1/2I+ \Kcal_{\p D_2}^*)[\Gvf_2] \big) = (k_2-1)\Gn_2.
$$
These integral equations can be expressed in a short form as follows:
\beq\label{Gvfg}
\left( \GL - \Kbb_{\p D}^* \right) [\Gvf] = \Gn,
\eeq
where
\beq
\GL:= \begin{bmatrix} \Gl_1 I & 0 \\ 0 & \Gl_2 I \end{bmatrix}
\eeq
($\Gl_j$ is the constant defined by \eqnref{Glj} and $I$ is the identity operator), and
\beq
\Kbb_{\p D}^* \begin{bmatrix} \Gvf_1 \\ \Gvf_2 \end{bmatrix} = \begin{bmatrix} \Kcal_{\p D_1}^*[\Gvf_1] & \p_{\nu} \Scal_{\p D_2}[\Gvf_2]|_{\p D_1} \\
\p_{\nu} \Scal_{\p D_1}[\Gvf_1]|_{\p D_2}  & \Kcal_{\p D_2}^*[\Gvf_2] \end{bmatrix}.
\eeq
The operator $\Kbb_{\p D}^*$ is the NP operator associated with $\p D=\p D_1 \cup \p D_2$.

Let, for $i,j=1,2$, $S_{ij}$ be the operator defined by
\beq
S_{ij} [\Gm]:= \Scal_{\p D_j}[\Gm] |_{\p D_i}, \quad \Gm \in H^{-1/2}(\p D_j),
\eeq
where the integration is carried out over $\p D_j$ and the integral is evaluated on $\p D_i$. Then $S_{ij}$ maps $H^{-1/2} (\p D_j)$ into $H^{1/2} (\p D_i)$. We then define
\beq
\Sbb_{\p D}: = \begin{bmatrix} S_{11} & S_{12} \\ S_{21} & S_{22} \end{bmatrix} .
\eeq
Then $\Sbb_{\p D}$ is a bounded operator from $\Hcal^{-1/2}$ into $\Hcal^{1/2}:= H^{1/2} (\p D_1) \times H^{1/2} (\p D_2)$.

The following analogy of the Plemelj's symmetrization principle was proved in \cite{ACKLM13}:
\beq\label{Ple}
\Sbb_{\p D} \Kbb_{\p D}^* = \Kbb_{\p D} \Sbb_{\p D}.
\eeq
We define a bilinear form on $\Hcal^{-1/2}$ by
\beq\label{biform}
\la \Gvf, \psi \ra_{\p D}:= - (\Gvf, \Sbb_{\p D}[\psi]),
\eeq
where the right-hand side is the duality pairing between $\Hcal^{-1/2}$ and $\Hcal^{1/2}$. Then $\la \cdot, \cdot \ra_{\p D}$ is actually an inner product on $\Hcal_0^{-1/2}$ and $\Kbb_{\p D}^*$ is compact and self-adjoint on $\Hcal^{-1/2}$, namely,
$$
\la \Gvf, \Kbb_{\p D}^*[\psi] \ra_{\p D} = \la \Kbb_{\p D}^*[\Gvf], \psi \ra_{\p D},
$$
which is due to \eqnref{Ple}.

Then we have the following lemma. It is convenient to use the following notation: for a function $u$ on $\Rbb^2$, $\p_\nu u|_{\p D_i \pm}$ denotes the trace of $\p_\nu u$ on $\p D_i$ from outside and inside of $D_i$, respectively, and
\beq
\p u|_{\pm} := \begin{bmatrix} \p_\nu u|_{\p D_1 \pm} \\
\p_\nu u|_{\p D_2 \pm} \end{bmatrix}.
\eeq
If $u(x)= \Scal_{\p D}[\Gvf](x)$ for some $\Gvf \in \Hcal^{-1/2}$, then
\beq
\p u|_\pm = (\pm 1/2I + \Kbb_{\p D}^*)[\Gvf].
\eeq

\begin{lemma}\label{lem:1/2}
Let
\beq
E_{1/2}:= \{ \psi \in \Hcal^{-1/2} : \la \psi, \Gvf \ra_{\p D}=0 \text{ for all $\Gvf \in \Hcal_0^{-1/2}$} \},
\eeq
so that $\Hcal^{-1/2}= \Hcal_0^{-1/2} \oplus E_{1/2}$.
\begin{itemize}
\item[(i)] Eigenvalues of $\Kbb_{\p D}^*$ on $\Hcal_0^{-1/2}$ belongs to $(-1/2,1/2)$.
\item[(ii)] $\GL - \Kbb_{\p D}^*$ is invertible on $\Hcal_0^{-1/2}(\p D)$.
\item[(iii)] $E_{1/2}$ is the eigenspace of $\Kbb_{\p D}^*$ corresponding to the eigenvalue $1/2$ and consists of $\psi \in \Hcal^{-1/2}$ such that $\Scal_{\p D}[\psi]$ is constant on $D_1$ and $D_2$. In particular, $E_{1/2}$ is of two dimensions.
\end{itemize}
\end{lemma}
\begin{proof}
(i) is well-known (see, for example, \cite[Chapter XI, Section 11]{Kellog-book}).

Since $\Kbb_{\p D}^*$ is compact, (ii) follows from the injectivity of $\GL - \Kbb_{\p D}^*$ on $\Hcal_0^{-1/2}(\p D)$ by Fredholm alternative. If $\left( \GL - \Kbb_{\p D}^* \right) [\Gvf]=0$, then $u(x) := \Scal_{\p D}[\Gvf](x)$ is the solution to \eqnref{mainproblem} with $\Gn=0$ and $u(x)=O(|x|^{-1})$ as $|x| \to \infty$. It implies that $u \equiv 0$ and hence $\Gvf=0$.

If $\psi \in E_{1/2}$, then $(\Gvf, \Sbb_{\p D}[\psi])=0$ for all $\Gvf \in \Hcal_0^{-1/2}$. Thus $\Sbb_{\p D}[\psi]$ is constant, that is, $u:= \Scal_{\p D}[\psi]$ is constant on $D_1$ and $D_2$. Therefore, we have
$$
\left( -1/2 I + \Kbb_{\p D}^* \right) [\psi] = \p u|_- =0 ,
$$
and hence $\Kbb_{\p D}^* [\psi] = 1/2 \psi$. The converse can be proved by reversing the argument.
\end{proof}

\subsection{$1/2$ eigenvectors}\label{subsec:1/2}

Let
$q^1$ be the solution to
\beq\label{qonedef}
\begin{cases}
\ds \GD q^1  = 0 \quad & \mbox{in } \Rbb^2 \setminus (\p D_1 \cup \p D_2),\\
\ds q^1 = \mbox{constant} \quad &\mbox{on } \p D_i, \ i=1,2,  \\
\nm
\ds \int_{\p D_i} \p_\nu q^1|_+ \, d\Gs = (-1)^{i+1} , \quad &i=1,2, \\
\nm
\ds q^1(x) =O(|x|^{-1}) \quad &\mbox{as } |x|\rightarrow \infty.
\end{cases}
\eeq
The existence of the solution $q^1$ (and $q^2$ below) is proved in \cite{ACKLY13} for inclusions $D_j$ of general shape, not necessarily disks.  Note that the constant values $q^1|_{\p D_j}$ satisfy that $q^1|_{\p D_1} <0$ and $q^1|_{\p D_2} >0$, which can be shown using the maximum principle and Hopf's lemma. In particular, $q^1$ has a potential gap, namely,
\beq\label{qgap}
q^1|_{\p D_2} - q^1|_{\p D_1} >0.
\eeq
Let
\beq\label{mudef}
m= -\frac{q^1|_{\p D_2}}{q^1|_{\p D_1}} (>0),
\eeq
and let $q^2$ be the solution to
\beq\label{qtwodef}
\begin{cases}
\ds \GD q^2  = 0 \quad \mbox{in } \Rbb^2 \setminus (\p D_1 \cup \p D_2),\\
\ds q^2 = \mbox{constant} \quad \mbox{on } \p D_i, \ i=1,2,  \\
\nm
\ds \int_{\p D_1} \p_\nu q^2|_+ \, d\Gs = m, \quad \int_{\p D_2} \p_\nu q^2|_+ \, d\Gs = 1, \\
\ds q^2 (x)= \frac{m+1}{2\pi} \ln|x| + O(|x|^{-1}) \quad\mbox{as }|x| \to \infty.
\end{cases}
\eeq

Let
\beq
\psi^i:=
\begin{bmatrix}
\p_\nu q^i|_{\p D_1+} \\ \p_\nu q^i|_{\p D_2+}
\end{bmatrix}, \quad j=1,2.
\eeq
One can see from \eqnref{qonedef} and \eqnref{qtwodef} that for $i=1,2$
$$
q^i(x)= \Scal_{\p D}[\psi^i](x)+ \Dcal_{\p D}[q^i](x)
$$
for $x \in \Rbb^2 \setminus (D_1 \cup D_2)$ by applying Green's theorem, where $\Dcal_{\p D}$ is the double layer potential, namely,
$$
\Dcal_{\p D} [\Gvf] (x):=
\frac{1}{2\pi}\int_{\p D} \frac{(x-y) \cdot \nu_x}{|x-y|^2} \; \Gvf(y)\,d\Gs(y), \quad x \in \Rbb^2 \setminus (\p D_1 \cup \p D_2).
$$
Since $q^i$ is constant on $\p D_j$ for $j=1,2$, $\Dcal_{\p D}[q^i](x)=0$, and hence $q^i(x)= \Scal_{\p D}[\psi^i](x)$ for $x \in \Rbb^2 \setminus (D_1 \cup D_2)$. We then infer from Lemma \ref{lem:1/2} that $\psi^i$ belongs to the eigenspace $E_{1/2}$.

The eigenfunctions $\psi^1$ and $\psi^2$ constitute an orthogonal basis of $E_{1/2}$.
In fact, we have
\begin{align*}
\la \psi^2 , \psi^1 \ra_{\p D} &= -(\psi^2, \Sbb_{\p D}[\psi^1]) \\
&= - q^1|_{\p D_1} \int_{\p D_1} \psi^2_1 d\Gs -q^1|_{\p D_2} \int_{\p D_2} \psi^2_2 d\Gs = -m q^1|_{\p D_1} -  q^1|_{\p D_2}=0,
\end{align*}
where the last equality follows from the definition \eqnref{mudef} of $m$.
Note that
$$
0= (\psi^1, \Sbb_{\p D}[\psi^2]) = q^2|_{\p D_1} \int_{\p D_1} \psi^1_1 \, d\Gs + q^2|_{\p D_2} \int_{\p D_2} \psi^1_2 \, d\Gs = q^2|_{\p D_1} -  q^2|_{\p D_2}.
$$
Thus we have
\beq\label{nogap}
q^2|_{\p D_1} = q^2|_{\p D_2},
\eeq
that is, $q^2$ has no potential gap on $\p D_1$ and $\p D_2$, from which one can infer that $\nabla q^2$ is bounded in $\Rbb^2 \setminus (D_1 \cup D_2)$ (this can be proved by following the same lines of the proof of Theorem 2.1 in \cite{KLY}). It is also known that $|\nabla q^1| \gtrsim \Ge^{-1/2}$ (see \cite{Yun} or the end of this subsection).

One can easily see that
\beq\label{psionenorm}
\| \psi^1 \|_{\p D}^2 = q^1|_{\p D_2} - q^1|_{\p D_1},
\eeq
and
\beq\label{psitwonorm}
\| \psi^2 \|_{\p D}^2 =  - q^2|_{\p D_1} (1+ m)= q^2|_{\p D_1} \left( \frac{q^1|_{\p D_2}}{q^1|_{\p D_1}} -1 \right).
\eeq
By \eqnref{qgap}, $\| \psi^1 \|_{\p D} >0$. But we do not know if $q^2|_{\p D_1} \neq 0$. If it is the case, then $\la \cdot, \cdot \ra_{\p D}$ is an inner product not only on $\Hcal_0^{-1/2}$ but also on $\Hcal^{-1/2}$.

What we wrote so far in this subsection holds for inclusions of arbitrary shape (with Lipschitz continuous boundaries). If $D_1$ and $D_2$ are disks, then $q^1$ is explicitly given by
\beq\label{qBdef}
q^1(x) = \frac{1}{2\pi} (\ln|x-p_1|-\ln|x-p_2|), \quad x \in \Rbb^2 \setminus (D_1 \cup D_2),
\eeq
where $p_1$ and $p_2$ be the unique fixed points of the combined reflections $R_1\circ R_2$ and $R_2\circ R_1$, respectively. Here, $R_j$ be the inversion with respect to the circle $\p D_j$. Since $\p D_1$ and $\p D_2$ are the Apollonius circles with respect to points $p_1$ and $p_2$, $q^1$ is constant on $\p D_1$ and $\p D_2$ (constants may be different depending on the radii of the circles). This function was used for the first time in \cite{Yun} for analysis of field concentration. From the formula \eqnref{qBdef}, one can see that $\nabla q^1$ blows up at the order of $\Ge^{-1/2}$ as $\Ge$ tends to zero.

\subsection{Well-posedness for the transmission problems}

Let
\beq\label{Ximatrix}
\Xi := \begin{bmatrix} k_1-1 & 0 \\ 0 & k_2-1 \end{bmatrix}.
\eeq
For $\Gn = (\Gn_1, \Gn_2)^T \in \Hcal^{-1/2}(\p D)$, let
$$
m_i := \frac{\la \Xi \Gn, \psi^i \ra_{\p D}}{\| \psi^i \|_{\p D}^2} , \quad i=1,2,
$$
assuming that $\| \psi^2 \|_{\p D} \neq 0$. One can see from \eqnref{psionenorm} that
\beq\label{Monedef}
m_1= \frac{- \sum_{j=1}^2 q^1|_{\p D_j} (k_j-1) \int_{\p D_j} \Gn_j}{q^1|_{\p D_2} - q^1|_{\p D_1}},
\eeq
and from \eqnref{nogap} and \eqnref{psitwonorm} that
\beq\label{Mtwodef}
m_2= \frac{q^1|_{\p D_1} \sum_{j=1}^2 (k_j-1) \int_{\p D_j} \Gn_j}{q^1|_{\p D_2} - q^1|_{\p D_1}}.
\eeq
We take \eqnref{Mtwodef} as the definition of $m_2$ when $\| \psi^2 \|_{\p D} = 0$, or equivalently $q^2|_{\p D_1} = 0$. We then have
\beq\label{E1/2compMj}
\Xi \Gn - (m_1 \psi^1 + m_2 \psi^2) \in \Hcal^{-1/2}_0(\p D).
\eeq

Let $u$ be the solution to \eqnref{mainproblem} and let $v:= u- (m_1 q^1 + m_2 q^2)$. Then $v$ is the solution to \eqnref{mainproblem} with $\Gn$ replace by $\Gn^0$, which is defined by
\beq\label{gzerodef}
\Gn^0:= \Gn- \Xi^{-1} (m_1 \psi^1 + m_2 \psi^2).
\eeq
We emphasize that $\Gn^0$ belongs to $\Hcal_0^{-1/2}(\p D)$.

\begin{prop}\label{prop:unique}
For $\Gn \in \Hcal^{-1/2}(\p D)$, the solution $v$ to \eqnref{mainproblem} is unique and given by
\beq\label{mainprobsol}
v(x)= m_1 q^1(x) + m_2 q^2(x) + \Scal_{\p D}[\Gvf^0](x), \quad x \in \Rbb^2,
\eeq
where $m_1, m_2$ are numbers given by \eqnref{Monedef} and \eqnref{Mtwodef}, and $\Gvf^0 \in \Hcal_0^{-1/2}(\p D)$ is the unique solution to
\beq\label{Gvfgzero}
\left( \GL - \Kbb_{\p D}^* \right) [\Gvf^0] = \Gn^0
\eeq
with $\Gn^0$ given by \eqnref{gzerodef}.
\end{prop}

\begin{proof}
It suffices to show that the solution is unique. For that, we prove that the solution $u$ to \eqnref{mainproblem} is zero assuming $\Gn=0$. Let $v$ be the solution to \eqnref{mainproblem} with $\Gn=0$. If $v (x)= c \ln |x| +O(|x|^{-1})$ for some constant $c$ as $|x| \to \infty$, let
$$
w(x):= v(x)- \frac{2\pi c}{m+1}q^2(x).
$$
Then, $w$ is a solution to \eqnref{mainproblem} with $\Gn= -\frac{2\pi c}{m+1} \left(\frac{1}{k_1-1} \p_\nu q^2|_{\p D_1+}, \frac{1}{k_2-1} \p_\nu q^2|_{\p D_2+}\right)^T$ and $w(x)= O(|x|^{-1})$ as $|x| \to \infty$ by the fourth line of \eqnref{qtwodef}. It then follows from the third line of \eqnref{qtwodef} that
$$
\sum_{j=1}^2 \int_{\p D_j} (\p_\nu w|_+ - k_j \p_\nu w|_-) = -\frac{2\pi c}{m+1} \sum_{j=1}^2 \int_{\p D_j} \p_\nu q^2|_{+} = -2\pi c.
$$
On the other hand, we have from divergence theorem that
$$
\sum_{j=1}^2 \int_{\p D_j} (\p_\nu w|_+ - k_j \p_\nu w|_-) = -\int_{\Rbb^2 \setminus D} \GD w - \sum_{j=1}^2 k_j \int_{D_j} \GD w =0.
$$
Thus, we have $c=0$, and hence $v (x)= O(|x|^{-1})$ as $|x| \to \infty$ from which we conclude that $v=0$.
\end{proof}

We have the following corollary:
\begin{cor}\label{cor:repin}
Let $F, \Gn^0, m_1, m_2$ be quantities respectively given by \eqnref{Fdef}, \eqnref{gzerodef}, \eqnref{Monedef}, and \eqnref{Mtwodef}, and let $\Gvf^0 \in \Hcal_0^{-1/2}(\p D)$ be the solution to \eqnref{Gvfgzero}.
Then, the unique solution $u$ to \eqnref{NCE} is given by
\beq\label{Nmainprobsol}
u= F + m_1 q^1 + m_2 q^2 + \Scal_{\p D}[\Gvf^0].
\eeq
\end{cor}

If $\Gn_j= \p_\nu H|_{\p D_j}$ for $j=1,2$ where $H$ is the harmonic function appearing in \eqnref{HCE}, then
$$
\int_{\p D_j} \Gn_j = \int_{\p D_j} \p_\nu H = 0.
$$
Thus, in this case, we have $m_1=m_2=0$.

\begin{cor}\label{cor:repho}
Let $\Gvf \in \Hcal_0^{-1/2}(\p D)$ be the solution to \eqnref{Gvfg} with $\Gn_j= \p_\Gv H|_{\p D_j}$.
Then, the unique solution $u$ to \eqnref{HCE} is given by
\beq\label{Hmainprobsol}
u= H + \Scal_{\p D}[\Gvf].
\eeq
\end{cor}

\section{NP spectrum on two circles}\label{sec:twocircles}

From now on we suppose that $D_1$ and $D_2$ are disks of radii $r_1$ and $r_2$, respectively. In this section we recall a complete spectrum of the NP operator on $\Hcal^{-1/2}_0$ using the NP spectrum on concentric disks and M\"obius invariance of the NP spectrum (see, for example, \cite{Sch1}). NP spectrum on two disks is derived in \cite{BT} and we follow the method there. However, we need to review the method in some detail since we need to prove M\"obius invariance of the inner product (Lemma \ref{single_star} (ii)), and to derive the solution formula for the integral equation \eqnref{Gvfg} (Lemma \ref{lem:intsol}).

\subsection{NP spectrum on concentric circles}

Let $D_j^*$, $j=1,2$ be the disk of radius $R_j$ centered at $0$, where $R_1 < R_2$, and let $\Kbb_{\p D^*}$ be the NP operator on $\p D^*$.
The complete spectrum of $\Kbb_{\p D^*}$ on $\Hcal_0^{-1/2}(\p D^*)$ is known (see \cite{ACKLM13}) as we describe in the following lemma.

\begin{lemma}\label{starspec}
Let $\Gr:= R_1/R_2$. The eigenvalues of $\Kbb_{\p D^*}$ on $\Hcal_0^{-1/2}(\p D^*)$ are $\pm \frac{1}{2} \Gr^{|n|}$ and the corresponding eigenfunctions are $f^{|n|,\pm}$ and $f^{-|n|,\pm}$ where
\beq\label{stareigenfunction}
f^{n,\pm} =  \left( \frac{1}{R_1} e^{in\Gt} , \pm \frac{1}{R_2} e^{in\Gt}  \right)^T , \quad n= \pm 1, \pm 2, \ldots.
\eeq
\end{lemma}

Actually Lemma \ref{starspec} is a consequence of the following formula whose proof can be found in \cite{ACKLM13}: if $\GO$ is the disk of radius $R$ (centered at $0$), then for $n \neq 0$
\beq\label{slayerdisk}
\Scal_{\p \GO} \left[\frac{1}{R} e^{in\Gt} \right](r, \Gt) = \begin{cases}
\ds -\frac{1}{2|n|} \left(\frac{r}{R}\right)^{|n|} e^{in\Gt} \quad &\text{if } r \leq R, \\
\ds -\frac{1}{2|n|} \left(\frac{R}{r}\right)^{|n|} e^{in\Gt} \quad &\text{if } r>R.
\end{cases}
\eeq
Using \eqnref{slayerdisk}, one can see that the following formula holds:
if $n >0$, then
\beq\label{single_conc_circ2}
\Scal_{\p D^*} [f^{n, \pm}] (\Gz) =
\begin{cases}
\ds -\frac{1}{2n} (R_1^{-n} \pm R_2^{-n}) \Gz^n \quad & \text{if } |\Gz| \le R_1 , \\
\nm
\ds -\frac{1}{2n} (R_1^n \bar{\Gz}^{-n} \pm R_2^{-n} \Gz^n ) \quad & \text{if } R_1 < |\Gz| \leq R_2, \\
\nm
\ds -\frac{1}{2n} (R_1^{n} \pm R_2^{n}) \bar{\Gz}^{-n} \quad & \text{if } R_2 < |\Gz|,
\end{cases}
\eeq
and if $n <0$, then
\beq\label{single_conc_circ3}
\Scal_{\p D^*} [f^{n, \pm}] (\Gz) = \ol{\Scal_{\p D^*} [f^{-n, \pm}] (\Gz)},
\eeq
where $\Gz=r e^{i\Gt}$. These formulas play a crucial role in what follows.

\subsection{M\"obius invariance and NP spectrum on two circles}

By translating and rotating $D$ if necessary, we may assume that their centers are located at $(c_1,0)$ and $(c_2,0)$, where
\beq\label{conedef}
c_1 = \frac{r_2^2 - r_1^2 - (r_1+r_2+\Ge)^2}{2(r_1 + r_2 + \Ge)} - \frac{\Gb}{2} , \quad c_2 = c_1 + r_1 + r_2 + \Ge
\eeq
with
\beq\label{Gbdef}
\Gb = \sqrt{\Ge} \frac{\sqrt{(2r_1 + \Ge)(2r_2 +\Ge)(2r_1 + 2r_2 + \Ge)}}{r_1 + r_2 + \Ge} ,
\eeq
where $\Ge$ is the separation distance between $D_1$ and $D_2$.  These numbers are chosen so that $\p D_1$ and $\p D_2$ are mapped onto two concentric circles by an inversion in a circle and a translation:
if we let
\beq\label{zstardef}
z^*=Tz := \frac{\Gb}{z} + 1,
\eeq
then $T(\p D_j)$ ($j=1,2$) is the circle of the radius $R_j$ centered at $0$, where $R_j$ is given by
\beq\label{Ronetwodef}
R_1^2 = 1+ \frac{\Gb}{c_1} , \quad
R_2^2 = 1+ \frac{\Gb}{c_2} .
\eeq
Let
\beq
D_1^* := T(D_1)= \{ |\Gz| < R_1 \}, \quad D_2^* := T(D_2)=  \{ |\Gz| > R_2 \}.
\eeq

One can see from \eqnref{Gbdef} that
\beq\label{Gbasym}
\Gb= \frac{4}{r_*} \sqrt{\Ge} + O(\Ge),
\eeq
where $r_*$ is defined by \eqnref{notation2}.
Since $c_1= -r_1 + O(\sqrt{\Ge})$ and $c_2= r_2 + O(\sqrt{\Ge})$, we see from \eqnref{Ronetwodef} that
\beq\label{Rest}
R_1 = 1 - \frac{\Gb}{2r_1} + O(\Ge), \quad
R_2 = 1 + \frac{\Gb}{2r_2} + O(\Ge),
\eeq
which together with \eqnref{Gbasym} yields
\beq\label{Grasymp}
\Gr := \frac{R_1}{R_2} = 1 - r^* \sqrt{\Ge} + O(\Ge),
\eeq
as $\Ge \to 0$.

Under the change of variables \eqnref{zstardef} the line and area elements are respectively transformed as
\beq\label{cvarea}
ds(z^*)= \frac{\Gb}{|z|^2} ds(z)
\eeq
and
\beq\label{cvarea2}
dA(z^*)= \frac{\Gb^2}{|z|^4} dA(z).
\eeq

Define the transformation $U$ from $\Hcal^{-1/2}(\p D)$ to $\Hcal^{-1/2}(\p D^*)$ by
\begin{align} \label{phistar}
(U\Gvf)(z^*) =\Gvf^*(z^*) := \frac{|z|^2}{\Gb} \Gvf(z), \quad z^* \in \p D^*.
\end{align}
We obtain the following lemma (a similar lemma has been proved in \cite{JK} for the case with a single inclusion).

\begin{lemma}\label{single_star}
\begin{itemize}
\item[{\rm (i)}] The following transformation formula holds for any $\Gvf \in \Hcal^{-1/2}_0(\p D)$:
\beq\label{singletrans}
\Scal_{\p D^*} [ \Gvf^* ](z^*)  = \Scal_{\p D} [\Gvf](z) - \Scal_{\p D} [\Gvf](0), \quad z \in \Cbb.
\eeq
\item [{\rm (ii)}] $U$ is a unitary transform from $\Hcal^{-1/2}_0(\p D)$ to $\Hcal^{-1/2}_0(\p D^*)$, namely,
\beq\label{unitary}
\la \Gvf^*, \psi^* \ra_{\p D^*} = \la \Gvf, \psi \ra_{\p D}
\eeq
for any $\Gvf, \psi \in \Hcal^{-1/2}_0(\p D)$.
\item [{\rm (iii)}] $\Kbb^*_{\p D}[\Gvf] = \Gl \Gvf$ ($|\Gl| < 1/2$) if and only if $\Kbb^*_{\p D^*}[\Gvf^*] = -\Gl \Gvf^*$. In fact, we have
\beq\label{KUUK}
\Kbb^*_{\p D^*} \circ U= - U \circ \Kbb^*_{\p D}.
\eeq
\end{itemize}
\end{lemma}
\begin{proof}
From the definition \eqref{zstardef} of $z^*$, we have
\begin{align}
|z^*-w^*| = \frac{\Gb}{|z||w|}|z-w|,
\end{align}
and hence
\begin{align} \label{gammastar}
\GG(z^* - w^*) = \GG(z - w)  - \GG(w) - \GG(z) + \GG(\Gb).
\end{align}
Here $\GG(z)= \frac{1}{2\pi} \ln|z|$. Hence for any $\Gvf \in \Hcal^{-1/2}_0(\p D)$, we have from \eqref{cvarea} and \eqref{gammastar} that
\begin{align*}
\Scal_{\p D^*}[\Gvf^*](z^*) &= \int_{\p D^*} \GG(z^* - w^*) \, \Gvf^*(w^*) \, d\Gs(w^*) \\
&=\int_{\p D} \left( \GG(z - w)  - \GG(w) - \GG(z) + \GG(\Gb)\right) \, \Gvf(w) \, d\Gs(w) \\
&= \Scal_{\p D}[\Gvf](z) - \Scal_{\p D}[\Gvf](0),
\end{align*}
which proves \eqref{singletrans}.

That $U$ is a unitary transform follows from \eqref{singletrans} immediately. In fact, we have
\begin{align*}
\la \Gvf^*, \psi^* \ra_{\p D^*} &= -\int_{\p D^*} \Gvf^*(z^*) \, \overline{\Scal_{\p D^*}[\Gy^*](z^*)} \, d\Gs(z^*) \\
&= -\int_{\p D} \Gvf(z) \; \overline{\left(\Scal_{\p D}[\Gy](z) - \Scal_{\p D}[\Gy](0)\right)} \, d\Gs(z) \\
&= \la \Gvf, \Gy \ra_{\p D},
\end{align*}
where the last equality holds since $\Gvf \in \Hcal^{-1/2}_0(\p D)$.

Let $n_z$ and $n_{z^*}$ be the complexified outward unit normal vectors on $\p D_j$ and $\p D^*_j$, respectively, namely,
$$
n_z = \frac{z-c_j}{r_j} \quad\mbox{if } z \in \p D_j,
$$
and
$$
n_{z^*} = (-1)^{j+1} \frac{z^*}{R_j^*} \quad\mbox{if } z^* \in \p D^*_j, \quad j=1,2.
$$
Then the following relation holds:
\begin{align}\label{nustar}
n_{z^*} = \frac{\bar{z}}{z} n_z.
\end{align}
In fact, since $|T(z)|^2=R_j$ if $z \in \p D_j$, we see that ${\bar\p_z} |T(z)|^2= {T(z)} \ol{\p_z T(z)}$ is normal to $\p D_j$. Thus,
$$
\frac{{\bar\p_z} |T(z)|^2}{|{\bar\p_z} |T(z)|^2|} = - \frac{T(z) z}{R_j \zbar} = c n_z, \quad \mbox{for } z \in \p D_j,
$$
where $c$ is either $1$ or $-1$. One can easily see that $c=-1$, and hence \eqnref{nustar} holds.

Note that the outward normal derivatives $\nu_{z^*}$ and $\nu_z$ are given by
$$
\p_{\nu_{z^*}}= 2\Re (n_{z^*} \p_{z^*}), \quad \p_{\nu_z}= 2\Re (n_{z} \p_{z}).
$$
We thus have from \eqref{zstardef} and \eqnref{nustar} that
\begin{align} \label{nderistar}
\p_{\nu_{z^*}} = 2\Re (n_{z^*} \p_{z^*}) = -\frac{|z|^2}{\Gb} 2\Re (n_{z} \p_{z}) = -\frac{|z|^2}{\Gb} \p_{\nu_z} .
\end{align}
Hence we have from \eqref{gammastar} and \eqref{nderistar} the following identity
\begin{align*}
\Kbb^*_{\p D^*}[\Gvf^*](z^*) &= \int_{\p D^*} \p_{\nu_{z^*}} \GG(z^* - w^*) \, \Gvf^*(w^*) \, d\Gs(w^*) \\
&= -\frac{|z|^2}{\Gb} \int_{\p D} \p_{\nu_z} \left( \GG(z - w) - \GG(z) \right) \, \Gvf(w) \, d\Gs(w) \\
&= -\frac{|z|^2}{\Gb} \Kbb^*_{\p D}[\Gvf](z)
\end{align*}
which proves \eqref{KUUK}. This completes the proof.
\end{proof}

Let $f^{n, \pm}$ be the eigenfunctions of $\Kbb_{\p D^*}^*$ as given in \eqnref{stareigenfunction} and let $\Gvf^{n, \pm}$, $n \in \Zbb \setminus \{0\}$ be functions such that
\beq\label{Gvfeigen}
U(\Gvf^{n, \pm})= f^{n, \pm}.
\eeq
We have the following corollary from Lemma \ref{starspec}. This corollary is known (see \cite{BT}).

\begin{cor}
The eigenvalues of $\Kbb_{\p D}^*$ are $\pm\frac{1}{2} \Gr^{|n|}$ and $1/2$, and
\beq\label{eigenconcir}
\Kbb_{\p D}^*[\Gvf^{n, \pm}] = \mp \frac{1}{2} \Gr^{|n|} \Gvf^{n, \pm}.
\eeq
\end{cor}

Because of \eqnref{Grasymp}, more and more eigenvalues are approaching to $\pm 1/2$ as $\Ge$ tends to $0$. This causes blow-up of the gradient (and higher derivatives) when conductivities of inclusions tend to $\infty$ or $0$. The phenomenon that eigenvalues are approaching to $\pm 1/2$ as $\Ge$ tends to $0$ occurs when inclusions are of general shapes, as proved in \cite{BT2}.

\begin{lemma}
For each $n \neq 0$,
\beq\label{fnnorm}
\Vert \Gvf^{n,\pm} \Vert^2_{\p D} = \Vert f^{n,\pm} \Vert^2_{\p D^*} =\frac{2\pi}{|n|} \left( 1 \pm \Gr^{|n|} \right).
\eeq
\end{lemma}
\begin{proof}
The first identity in \eqnref{fnnorm} holds since $U$ is unitary. According to \eqnref{single_conc_circ2}, we have
\begin{align*}
\Vert f^{n,+} \Vert^2_{\p D^*} &= - (f_{n,+}, \Sbb_{\p D^*}[f^{n,+}]) \\
&= \frac{1}{2|n|R_1} \int_{\p D_1^*} \left(  1 + \Gr^{|n|}  \right) d\Gt + \frac{1}{2|n|R_2} \int_{\p D_2^*} \left(  \Gr^{|n|} + 1 \right) d\Gt ,
\end{align*}
from which \eqnref{fnnorm} follows.
\end{proof}

We observe from the explicit forms \eqnref{stareigenfunction} of eigenfunctions $f^{n, \pm}$ and the formulas (\eqnref{single_conc_circ2} and \eqnref{single_conc_circ3}) of their single layer potentials on concentric circles that they are not only orthogonal to each other, but also satisfy the following property: if $n \neq m$, then
\beq \label{sepaorth}
\int_{\p D_j^*} f_j^{n, \pm} \overline{\Scal_{\p D^*}[f^{m, \pm}]} = \int_{\p D_j^*} f_j^{n, \pm} \overline{\Scal_{\p D^*}[f^{m, \mp}]} =0,
\eeq
and
\beq \label{sepa}
\int_{\p D_j^*} f_j^{n, \pm} \overline{\Scal_{\p D^*}[f^{n, +}]} = (-1)^{j+1} \int_{\p D_j^*} f_j^{n, \pm} \overline{\Scal_{\p D^*}[f^{n, -}]} \neq 0 ,
\eeq
for $j=1,2$. This property enables us to solve the integral equation \eqnref{Gvfg} by means of spectral decomposition.

\begin{lemma}\label{lem:intsol}
If
\beq
\Gn= \sum_{n \neq 0} \big( C_{n, +} \Gvf^{n, +} + C_{n, -} \Gvf^{n, -} \big),
\eeq
then the solution $\Gvf$ to \eqnref{Gvfg} is given by
\beq\label{Gvfsolexp}
\Gvf = \sum_{n \neq 0} ( a_{n, +} \Gvf^{n, +} + a_{n, -} \Gvf^{n, -}),
\eeq
where
\begin{align}
a_{n,+} &= \frac{2}{4\Gl_1\Gl_2 - \Gr^{2|n|}} \big( (\Gl_1 + \Gl_2 - \Gr^{|n|}) C_{n, +} + (\Gl_1-\Gl_2) C_{n, -} \big), \label{solan+} \\
a_{n,-} &= \frac{2}{4\Gl_1\Gl_2 - \Gr^{2|n|}} \big( -(\Gl_1-\Gl_2) C_{n, +} +  (\Gl_1 + \Gl_2 + \Gr^{|n|}) C_{n, -} \big). \label{solan-}
\end{align}
\end{lemma}

\begin{proof}
One can see from \eqnref{eigenconcir}, \eqnref{sepaorth} and \eqref{sepa} that the following holds:
\begin{align*}
(\Gl_1 + 1/2 \Gr^{|n|}) a_{n, +} + (\Gl_1 - 1/2 \Gr^{|n|}) a_{n, -} &=  C_{n, +} + C_{n, -} , \\
(\Gl_2 + 1/2 \Gr^{|n|}) a_{n, +} - (\Gl_2 - 1/2 \Gr^{|n|}) a_{n, -} &=  C_{n, +} - C_{n, -}.
\end{align*}
Then, \eqnref{solan+} and \eqnref{solan-} immediately follow.
\end{proof}

If $\Gl_1=\Gl_2=\Gl$, then \eqnref{solan+} and \eqnref{solan-} are simplified to
$$
a_{n,\pm}= \frac{2C_{n, \pm}}{2\Gl \pm \Gr^{|n|}},
$$
as expected.

\section{Spectral representations of the solution}\label{sec:specrep}

In this section we derive representation formulas for solutions to the problems \eqnref{NCE} and \eqnref{HCE} in terms of NP spectrum.

\subsection{Vanishing integral condition}

We first deal with the case when the integrals of $f$ over $D_1$ and $D_2$ vanish, namely,
\beq\label{intzero}
\int_{D_1} f = \int_{D_2} f = 0.
\eeq
In this case, $\p_{\Gv} F \in \Hcal^{-1/2}_0(\p D_j)$ for $j=1,2$ because of the condition \eqnref{intzero}. In fact, since $\GD F= k_j^{-1} f$ on $D_j$ ($j=1,2$), we have
$$
\int_{\p D_j} \p_\nu F ds = \frac{1}{k_j} \int_{D_j} f=0.
$$
In particular, the constants $m_1$ and $m_2$ given in \eqnref{Monedef} and \eqnref{Mtwodef} become $0$. It thus follows from \eqnref{Nmainprobsol} that
the solution to \eqnref{NCE} admits the following representation:
\beq\label{uFv}
u(x)= F(x) + \Scal_{\p D} [\Gvf](x)
\eeq
where $\Gvf \in \Hcal^{-1/2}_0(\p D)$ is the solution to the integral equation
\beq\label{GvfpF}
\left( \GL - \Kbb_{\p D}^* \right) [\Gvf] = \p F.
\eeq

Since $\p_{\Gv} F \in \Hcal^{-1/2}_0(\p D_j)$ for $j=1,2$, the following Neumann boundary value problem admits a unique solution which we denote by $H_j$:
\beq\label{bvp}
\begin{cases}
\GD H_j = 0 \quad &\text{in } D_j, \\
\p_{\Gv} H_j = \p_{\Gv} F \quad &\text{on } \p D_j.
\end{cases}
\eeq
Let $h_j$ be the analytic function in $D_j^*$ such that $h_1(0)=0$, $\lim_{|\Gz| \to \infty} h_2(\Gz)=0$, and
\beq\label{hjdef}
H_j(z)= \Re (h_j \circ T)(z)  + C_j, \quad z \in D_j
\eeq
for some constant $C_j$,  where $T$ is the transformation defined in \eqnref{zstardef} and $\Re$ denotes the real part.

We look into the integral equation \eqnref{GvfpF}. We have from the Green's identity, the jump formula \eqref{singlejump2} and the relations \eqnref{cvarea}, \eqnref{Gvfeigen}, \eqnref{eigenconcir} that
\begin{align*}
\la \p F, \Gvf^{n,\pm} \ra_{\p D} &= - \sum_{j=1}^2 \int_{\p D_j} \p_{\Gv} F \; \overline{\Scal_{\p D}[\Gvf^{n,\pm}]} \\
&= - \sum_{j=1}^2 \int_{\p D_j} H_j \; \overline{\p_{\Gv} \Scal_{\p D}[\Gvf^{n,\pm}]|_-} \\
&= \frac{1}{2} \left( 1 \pm \Gr^{|n|} \right) \sum_{j=1}^2 \int_{\p D_j} H_j \; \overline{\Gvf_j^{n,\pm}} \\
&= \frac{1}{4} \left( 1 \pm \Gr^{|n|} \right) \sum_{j=1}^2 \int_{\p D^*_j} (h_j+ \ol{h_j})  \; \overline{f_j^{n,\pm}}.
\end{align*}
It then follows from \eqnref{fnnorm} that
\beq\label{Cnpmdef}
\frac{\la \p F, \Gvf^{n, \pm} \ra_{\p D}}{\Vert \Gvf^{n, \pm} \Vert_{\p D}^2} = \frac{|n|}{8\pi} C_{n, \pm}
:= \frac{|n|}{8\pi} \sum_{j=1}^2 \int_{\p D^*_j} (h_j+ \ol{h_j})  \; \overline{f_j^{n,\pm}}.
\eeq
Thus we have
\beq
\p F= \sum_{n \neq 0} \frac{|n|}{8\pi} \big( C_{n, +} \Gvf^{n, +} + C_{n, -} \Gvf^{n, -} \big).
\eeq

By Lemma \ref{lem:intsol}, the solution $\Gvf$ is given by
\beq\label{Gvfexp}
\Gvf= \sum_{n \neq 0} ( a_{n, +} \Gvf^{n, +} + a_{n, -} \Gvf^{n, -}),
\eeq
where
\begin{align}
a_{n,+} &= \frac{|n|}{4\pi( 4\Gl_1\Gl_2 - \Gr^{2|n|})} \big( (\Gl_1 + \Gl_2 - \Gr^{|n|}) C_{n, +} + (\Gl_1-\Gl_2) C_{n, -} \big), \label{an+} \\
a_{n,-} &= \frac{|n|}{4\pi( 4\Gl_1\Gl_2 - \Gr^{2|n|})} \big( -(\Gl_1-\Gl_2) C_{n, +} +  (\Gl_1 + \Gl_2 + \Gr^{|n|}) C_{n, -} \big). \label{an-}
\end{align}
The transformed function $\Gvf^*=U(\Gvf)$ is given by
\begin{align*}
\Gvf^* &= \sum_{n \neq 0} \frac{|n|}{4\pi( 4\Gl_1\Gl_2 - \Gr^{2|n|})} \Big[ C_{n, +} \Big( (\Gl_1 + \Gl_2 - \Gr^{|n|})f^{n, +} - (\Gl_1-\Gl_2)f^{n, -} \Big) \\
& \qquad + C_{n, -} \Big( (\Gl_1-\Gl_2) f^{n, +} + (\Gl_1 + \Gl_2 + \Gr^{|n|}) f^{n, -} \Big) \Big] .
\end{align*}
Thus, we have
\begin{align*}
\Scal_{\p D^*} [\Gvf^*] &= \sum_{n \neq 0} \frac{|n|}{4\pi( 4\Gl_1\Gl_2 - \Gr^{2|n|})} \Big[ C_{n, +} \Big( (\Gl_1 + \Gl_2 - \Gr^{|n|})H_{n, +} - (\Gl_1-\Gl_2) H_{n, -} \Big) \\
& \qquad + C_{n, -} \Big( (\Gl_1-\Gl_2) H_{n, +} + (\Gl_1 + \Gl_2 + \Gr^{|n|}) H_{n, -} \Big) \Big] .
\end{align*}
Here we use notation $H_{n, \pm}(\Gz):= \Scal_{\p D^*} [f^{n, \pm}](\Gz)$ for simplicity of expressions. Here and afterwards, we use $\Gz$ for $z^*=T(z)$.

Observe from the definition \eqnref{Cnpmdef} that $\ol{C_{n, \pm}}= C_{-n, \pm}$, and from \eqnref{single_conc_circ3} that $\ol{H_{n, \pm}}= H_{-n, \pm}$. We then have
\beq\label{SV1V2}
\Scal_{\p D^*} [\Gvf^*](\Gz) = \frac{V(\Gz) + \ol{V(\Gz)}}{2} ,
\eeq
where
\begin{align*}
V(\Gz) &= \sum_{n=1}^\infty \frac{n}{2\pi( 4\Gl_1\Gl_2 - \Gr^{2n})} \Big[ C_{n, +} \Big( (\Gl_1 + \Gl_2 - \Gr^{n})H_{n, +}(\Gz) - (\Gl_1-\Gl_2) H_{n, -}(\Gz) \Big) \\
& \qquad + C_{n, -} \Big( (\Gl_1-\Gl_2) H_{n, +}(\Gz) + (\Gl_1 + \Gl_2 + \Gr^{n}) H_{n, -}(\Gz) \Big) \Big] .
\end{align*}

Since $h_1$ is analytic in $|\Gz| < R_1$ and $h_1(0)=0$, it can be expressed as
\beq\label{honeexp}
h_1(\Gz) = \sum_{m=1}^{\infty} a_{1,m} \Gz^m , \quad |\Gz| \leq R_1
\eeq
for some coefficients $a_{1,m}$. The function $h_2$ can be expressed as
\beq\label{htwoexp}
h_2(\Gz) = \sum_{m=1}^{\infty} a_{2,m} \Gz^{-m} , \quad |\Gz| \ge R_2
\eeq
for some coefficients $a_{2,m}$. Then, if $n >0$, then one can see from the formula \eqnref{stareigenfunction} for $f^{n,\pm}$ that
$$
C_{n, \pm} = \sum_{j=1}^2 \int_{\p D^*_j} (h_j +\ol{h_j}) \; \overline{f_j^{n,\pm}} = 2\pi (a_{1,n} R_1^n \pm a_{2,n} R_2^{-n}).
$$
Thus $V(\Gz)$ can be written as
\beq\label{VB12}
V(\Gz) = - \frac{1}{2} \sum_{n=1}^\infty \frac{1}{4\Gl_1\Gl_2 - \Gr^{2n}} \left( B_{1,n} + B_{2,n} \right) ,
\eeq
where
\begin{align}
B_{1,n}(\Gz) &= a_{1,n} R_1^n \Big( (2\Gl_1 - \Gr^{n}) (-2n H_{n, +}(\Gz)) + (2 \Gl_2 + \Gr^{n}) (-2n H_{n, -}(\Gz)) \Big), \label{B1n} \\
B_{2,n}(\Gz) &= \ol{a_{2,n}} R_2^{-n} \Big( (2\Gl_2 - \Gr^{n})(-2n H_{n, +}(\Gz)) - (2 \Gl_1 + \Gr^{n}) (-2n H_{n, -}(\Gz)) \Big) . \label{B2n}
\end{align}

We see from \eqnref{single_conc_circ2} that
$$
-2n H_{n, \pm}(\Gz) =
\begin{cases}
\ds (R_1^{-1} \Gz)^n \pm (R_2^{-1} \Gz)^n \quad & \text{if } |\Gz| \le R_1 , \\
\nm
\ds (R_1^{-1} \Gzbar)^{-n} \pm (R_2^{-1} \Gz)^n \quad & \text{if } R_1 < |\Gz| \leq R_2, \\
\nm
\ds (R_1^{-1} \Gzbar)^{-n} \pm (R_2^{-1} \Gzbar)^{-n} \quad & \text{if } R_2 < |\Gz|.
\end{cases}
$$
Thus we have
\begin{align}
&B_{1,n}(\Gz) = 2a_{1,n} \times \nonumber \\
& \begin{cases}
\ds (\Gl_1+\Gl_2) \Gz^n + (\Gl_1-\Gl_2) (\Gr\Gz)^n - (\Gr^2 \Gz)^n,  \quad  & |\Gz| \le R_1 , \\
\nm
\ds (\Gl_1+\Gl_2) (R_1^2 \Gzbar^{-1})^n + (\Gl_1-\Gl_2) (\Gr\Gz)^n - (\Gr^2 \Gz)^n , \quad  & R_1 < |\Gz| \leq R_2, \\
\nm
\ds (\Gl_1+\Gl_2) (R_1^2 \Gzbar^{-1})^n + (\Gl_1-\Gl_2) (R_1R_2 \Gzbar^{-1})^n - (R_1^2 \Gzbar^{-1})^n , \quad & R_2 < |\Gz|,
\end{cases} \label{Bltwo}
\end{align}
and
\begin{align}
&B_{2,n}(\Gz) = 2 \ol{a_{2,n}} \times \nonumber \\
&
\begin{cases}
\ds (\Gl_1+\Gl_2) (R_2^2 \Gz^{-1})^{-n} - (\Gl_1-\Gl_2) (R_1R_2 \Gz^{-1})^{-n} - (R_2^2 \Gz^{-1})^{-n}, \quad  & |\Gz| \le R_1 , \\
\nm
\ds (\Gl_1+\Gl_2) (R_2^2 \Gz^{-1})^{-n} - (\Gl_1-\Gl_2) (\Gr^{-1} \Gzbar)^{-n} - (\Gr^{-2} \Gzbar)^{-n}, \quad & R_1 < |\Gz| \leq R_2, \\
\nm
\ds (\Gl_1+\Gl_2) \Gzbar^{-n} - (\Gl_1-\Gl_2) (\Gr^{-1} \Gzbar)^{-n} - (\Gr^{-2} \Gzbar)^{-n},  \quad  & R_2 < |\Gz|.
\end{cases}\label{Blthree}
\end{align}

Let, for $j=1,2$,
\beq
A_j(\Gz):=-\frac{1}{2} \sum_{n=1}^\infty \frac{B_{j,n}(\Gz)}{4\Gl_1\Gl_2 - \Gr^{2n}} .
\eeq
Then we have from \eqnref{VB12} that
\beq\label{A1A2}
V(\Gz) = A_1(\Gz)+ A_2(\Gz).
\eeq
It follows from \eqnref{Bltwo} that
\beq\label{A1def}
A_1(\Gz)
= \begin{cases}
\ds (\Gl_1+\Gl_2) w_1(\Gz) & \\
\qquad + (\Gl_1-\Gl_2) w_1(\Gr\Gz) - w_1(\Gr^2 \Gz), \quad & |\Gz| \le R_1 , \\
\nm
\ds (\Gl_1+\Gl_2) w_1(R_1^2 \Gzbar^{-1}) & \\
\qquad + (\Gl_1-\Gl_2) w_1(\Gr\Gz) - w_1(\Gr^2 \Gz), \quad & R_1 < |\Gz| \leq R_2, \\
\nm
\ds (\Gl_1+\Gl_2) w_1(R_1^2 \Gzbar^{-1}) & \\
\qquad + (\Gl_1-\Gl_2) w_1(R_1R_2 \Gzbar^{-1}) - w_1(R_1^2 \Gzbar^{-1}),  \quad & R_2 < |\Gz|,
\end{cases}
\eeq
where
\beq\label{w1def}
w_1(\Gz) = \sum_{n=1}^\infty \frac{a_{1,n} \Gz^n}{4\Gl_1\Gl_2 - \Gr^{2n}} , \quad |\Gz| \le R_1,
\eeq
and from \eqnref{Blthree} that
\beq\label{A2def}
A_{2}(\Gz) =
\begin{cases}
\ds (\Gl_1+\Gl_2) w_2(R_2^2 \Gz^{-1}) & \\
\qquad - (\Gl_1-\Gl_2) w_2(R_1R_2 \Gz^{-1}) - w_2(R_2^2 \Gz^{-1}), \quad  & |\Gz| \le R_1 , \\
\nm
\ds (\Gl_1+\Gl_2) w_2(R_2^2 \Gz^{-1}) & \\
\qquad - (\Gl_1-\Gl_2) w_2(\Gr^{-1} \Gzbar) - w_2(\Gr^{-2} \Gzbar), \quad & R_1 < |\Gz| \leq R_2, \\
\nm
\ds (\Gl_1+\Gl_2) w_2(\Gzbar) & \\
\qquad - (\Gl_1-\Gl_2) w_2(\Gr^{-1} \Gzbar) - w_2(\Gr^{-2} \Gzbar),  \quad  & R_2 < |\Gz|,
\end{cases}
\eeq
where
\beq\label{w2def}
w_2(\Gz) = \sum_{n=1}^\infty \frac{a_{2,n} \Gz^{-n}}{4\Gl_1\Gl_2 - \Gr^{2n}} , \quad |\Gz| \ge R_2.
\eeq

The functions $w_j$ can be expressed in terms of $h_j$. In fact, since
$$
\frac{1}{4\Gl_1\Gl_2 - \Gr^{2n}} = \sum_{l=0}^\infty  \frac{\Gr^{2ln}}{(4\Gl_1\Gl_2)^{l+1}},
$$
we have
\beq\label{w1def2}
w_1(\Gz) = \sum_{l=0}^\infty \frac{1}{(4\Gl_1\Gl_2)^{l+1}} \sum_{n=1}^\infty a_{1,n} (\Gr^{2l}\Gz)^n = \sum_{l=0}^\infty \frac{h_1 (\Gr^{2l}\Gz) }{(4\Gl_1\Gl_2)^{l+1}} , \quad |\Gz|<R_1.
\eeq
Likewise we have
\beq\label{w2def2}
w_2(\Gz) = \sum_{l=0}^\infty \frac{h_2 (\Gr^{-2l}\Gz) }{(4\Gl_1\Gl_2)^{l+1}}, \quad |\Gz|> R_2.
\eeq

So far, we derived the representation formula for the solution as summarized in the following proposition.

\begin{prop}\label{prop:rep}
Suppose that $f$ satisfies \eqnref{intzero}. The solution $u$ to the inhomogeneous problem \eqnref{NCE} is represented as
\beq\label{firstrep}
u(z)= F(z) + \Re \left( A_{1}(T(z)) + A_2(T(z)) \right) + \mbox{const.},
\eeq
where $F, A_1, A_2$ are given by \eqnref{Fdef}, \eqnref{A1def}, \eqnref{A2def}, respectively, with $w_1$ and $w_2$ given by \eqnref{w1def2} and \eqnref{w2def2}.
\end{prop}

The solution $u$ to the homogeneous problem \eqnref{HCE} admits the same spectral representation formula as \eqnref{firstrep}. In this case, the solution $H_j$ of the boundary value problem \eqnref{bvp} is replaced with the given harmonic function $H$.

\subsection{General case}
We now deal with the more general case when $f$ does not necessarily satisfy the condition \eqnref{intzero}.

We still assume that the centers $c_1, \ c_2$ of $D_1, \ D_2$ lie on the real axis and satisfy $c_1<c_2$. Suppose that $c_1=0$ for convenience. Let
$$
U(r,\Gt)=
\begin{cases}
\ds \frac{k_1+1}{2k_1} r\cos \Gt - \frac{(k_1-1)r_1^2}{2k_1} r^{-1} \cos \Gt, \quad &r>r_1, \\
\nm
\ds \frac{1}{k_1} r\cos\Gt, \quad & r \le r_1,
\end{cases}
$$
which satisfies
\beq\label{Ueqn}
\begin{cases}
\GD U = 0  \quad&\mbox{in } \Rbb^2 \setminus \ol{D_1}, \\
\GD U = 0  \quad&\mbox{in } D_1, \\
\ds U|_+ - U|_- = 0  \quad&\mbox{on }\p D_1 \\
\ds \p_\nu U|_+ - k_1 \p_\nu U|_- = 0  \quad&\mbox{on }\p D_1, \\
\ds U (x) = \frac{k_1+1}{2k_1} r\cos \Gt + O( |x|^{-1})~&\mbox{as }|x|\rightarrow\infty.
\end{cases}
\eeq
Let $\GT(\Gt)$ be a smooth function depending only on $\Gt$. We assume that $\GT(\Gt)$ supported in $(\pi/2, 3\pi/2)$ and satisfies
\beq\label{3000}
r_1 \int_0^{2\pi} G(\Gt) \cos \Gt d\Gt =1.
\eeq
Let $R(r)$ be a compactly supported smooth function of $r$ such that $R =1$ on $[0,r_1]$.
Let
$$
V_1(r, \Gt):= \GT(\Gt)R(r) U(r,\Gt),
$$
and
$$
v_1(r, \Gt) := \begin{cases}
\nabla V_1, \quad &r>r_1, \\
k_1 \nabla V_1, \quad & r \le r_1.
\end{cases}
$$

\begin{lemma}\label{V1v1}
The function $V_1$ is supported in $\{ x_1 < c_1 \}$, is the solution to
\beq\label{V1eqn}
\begin{cases}
	\nabla \cdot \Gs \nabla V_1 = \nabla \cdot v_1 \quad &\text{in } \Rbb^2, \\
	V_1(x) = O(|x|^{-1}) \quad &\text{as } |x| \to \infty,
\end{cases}
\eeq
and satisfies
\beq\label{V1est}
\| V_1 \|_{C^n(\ol{D_1})} \lesssim \frac{1}{k_1}, \quad \| V_1 \|_{C^n(\Rbb^2 \setminus D_1)} \lesssim \frac{k_1+1}{k_1}
\eeq
for all positive integer $n$. Moreover, $v_1$ satisfies
\beq\label{v1est}
\| v_1 \|_{C^n(\ol{D_1})} \lesssim 1, \quad \| v_1 \|_{C^n(\Rbb^2 \setminus D_1)} \lesssim \frac{k_1+1}{k_1}
\eeq
for all positive integer $n$ and
\beq\label{v1int}
\int_{D_1} \nabla \cdot v_1=1.
\eeq
\end{lemma}
\begin{proof}
It follows from the third and fourth lines of \eqnref{Ueqn} that 
$$
V_1|_+ - V_1|_- = 0
$$ 
and 
$$
\p_\nu V_1|_+ - k_1 \p_\nu V_1|_- = \GT(\Gt) (\p_\nu U|_+ - k_1 \p_\nu U|_- )= 0  
$$
on $\p D_1$. Since $V_1=0$ on $\{ x_1 \ge c_1 \}$, we have $V_1|_+ - V_1|_- = 0$ and $\p_\nu V_1|_+ - k_2 \p_\nu V_1|_- =0$ on $\p D_2$. Thus $V_1$ satisfies \eqnref{V1eqn}. Estimates \eqnref{V1est} and \eqnref{v1est} immediately follow from the definitions of $V_1$ and $v_1$. For \eqnref{v1int}, we have
$$
\int_{D_1} \nabla \cdot v_1 = k_1 \int_{\p D_1} \p_\nu V_1 = \int_{\p D_1} \GT(\Gt) \cos \Gt d\Gs=1,
$$
where the last equality holds because of \eqnref{3000}.
\end{proof}

Likewise, we construct the functions $V_2$ such that $V_2$ and $v_2$, defined by
$$
v_2(r, \Gt) := \begin{cases}
\nabla V_2, \quad &\mbox{in } D_2, \\
k_2 \nabla V_2, \quad &\mbox{in } \Rbb^2 \setminus D_2,
\end{cases}
$$
satisfy the following lemma.

\begin{lemma}\label{V2v2}
The function $V_2$ is supported in $\{ x_1 > c_2 \}$, is the solution to
\beq\label{V2eqn}
\begin{cases}
	\nabla \cdot \Gs \nabla V_2 = \nabla \cdot v_2 \quad &\text{in } \Rbb^2, \\
	V_2(x) = O(|x|^{-1}) \quad &\text{as } |x| \to \infty,
\end{cases}
\eeq
and satisfies
\beq\label{V2est}
\| V_2 \|_{C^n(\ol{D_2})} \lesssim \frac{1}{k_2}, \quad \| V_2 \|_{C^n(\Rbb^2 \setminus D_2)} \lesssim \frac{k_2+1}{k_2}
\eeq
for all positive integer $n$. Moreover, $v_2$ satisfies
\beq\label{v2est}
\| v_2 \|_{C^n(\ol{D_2})} \lesssim 1, \quad \| v_2 \|_{C^n(\Rbb^2 \setminus D_2)} \lesssim \frac{k_2+1}{k_2}
\eeq
for all positive integer $n$ and
\beq\label{v2int}
\int_{D_2} \nabla \cdot v_2=1.
\eeq
\end{lemma}

For a given $f=\nabla \cdot g$, let
\beq\label{fzero}
g_0:= g- \left( \int_{D_1} f \right) v_1 - \left( \int_{D_2} f \right) v_2 .
\eeq
Then $f_0=\nabla \cdot g_0$ satisfies \eqnref{intzero} and we arrive at the following proposition.

\begin{prop}\label{thm: decomp}
For a given $f=\nabla \cdot g$, let $f_0=\nabla \cdot g_0$ be defined by \eqnref{fzero}. The solution $u$ to \eqnref{NCE} can be decomposed as
\beq\label{usol2}
u= \left( \int_{D_1} f \right) V_1 + \left( \int_{D_2} f \right) V_2 + u_0,
\eeq
where $u_0$ is the solution to \eqnref{NCE} with $f=f_0$ which is of the form \eqnref{firstrep}.
\end{prop}

\section{Proofs of theorems}\label{sec:pfs}

In this subsection we prove Theorem \ref{thm:main1}-\ref{thm:main5}. Thanks to Proposition \ref{thm: decomp}, for the case of the problem \eqnref{NCE} we may assume that \eqnref{intzero} holds. In fact, since
$$
\left| \int_{D_j} f \right| = \left| \int_{\p D_j} g \cdot \nu \right| \lesssim \| g \|_{L^\infty(\p D_j)}
$$
for $j=1,2$, we have
$$
\| g_0 \|_{n,\Ga}^* \lesssim \| g \|_{n,\Ga}^*.
$$
Moreover, we infer from \eqnref{V1est}, \eqnref{V2est}, and \eqnref{usol2} that
$$
\| u \|_{n,0} \lesssim \frac{1}{k_1} \| g \|_{L^\infty(\p D_1)} + \frac{1}{k_2} \| g \|_{L^\infty(\p D_2)} + \| u_0 \|_{n,0}.
$$
Thus it is enough to prove Theorem \ref{thm:main1} and \ref{thm:main3} with $u$ and $g$ replaced with $u_0$ and $g_0$.

In the rest of this section, we assume that \eqnref{intzero} holds for the problem \eqnref{NCE}.
We begin by proving a preliminary lemma, for which we introduce the following transformations: With constants $\Gb$ and $\Gr$ appearing in \eqnref{conedef} and \eqnref{Grasymp}, let, for $\Gr^2 \le t \le 1$ and $l=0,1,2, \ldots$,
\beq\label{Phidef}
\Phi_{l,t}(z):= T^{-1} (t \Gr^{2l} T(z))= \frac{\Gb z}{\Gb t \Gr^{2l} - (1-t\Gr^{2l}) z}, \quad z \in D_1
\eeq
and
\beq\label{Psidef}
\Psi_{l,t}(z):= T^{-1}(t \Gr^{-2l-1} T(z))= \frac{\Gb z}{\Gb t \Gr^{-2l-1} - (1-t\Gr^{-2l-1}) z}, \quad z \in D_2.
\eeq
Note that $\Phi_{l,t}$ maps $D_1$ into $D_1$ and $\Psi_{l,t}$ maps $D_2$ into $D_2$.

\begin{lemma}\label{lem:est1}
For $j=1,2$, let $V_j$ be an analytic function defined in $D_j$ such that $V_j \in C^{n}(\ol{D_j})$ for some positive integer $n$. For $\Gr^2 \le t \le 1$ and $l=0,1,2, \ldots$, let $v_1(z) = (V_1 \circ \Phi_{l,t})(z)$ for $z \in D_1$ and $v_{2}(z) = (V_2 \circ \Psi_{l,t})(z)$ for $z \in D_2$. There is a constant $C$ independent of $t$ and $l$ such that
\beq\label{vjlest}
\| v_{j} \|_{C^{n}(\ol{D_j})} \le C \Gr^{2l} (l+1)^{n} \| V_j \|_{C^{n}(\ol{D_j})}
\eeq
for $j=1,2$.
\end{lemma}

\begin{proof}
One can easily see from \eqnref{Phidef} that
\beq\label{Phideri}
\Phi_{l,t}^{(m)}(z) = \frac{m! \Gb^2 t\Gr^{2l} (1- t\Gr^{2l})^{m-1}}{(\Gb t\Gr^{2l} - (1-t\Gr^{2l}) z)^{m+1}}
\eeq
for $m=1,2, \ldots$, which can be written in terms of the $\Gz=T(z)$ variable as
\beq\label{Phideri2}
\Phi_{l,t}^{(m)}(z) = \Gb^{1-m} t\Gr^{2l} (1-t\Gr^{2l})^{m-1} \frac{(\Gz-1)^{m+1}}{(t\Gr^{2l} \Gz - 1)^{m+1}}.
\eeq
Since $|\Gz| \le R_1 <1$ and $0< t\Gr^{2l}<1$, we have
\beq\label{4209}
|\Gz-1| \le 2 |t\Gr^{2l} \Gz - 1 |,
\eeq
and hence
\beq\label{4210}
|\Phi_{l,t}^{(m)}(z)| \le 2^{m+1} \Gb^{1-m} t\Gr^{2l} (1-t\Gr^{2l})^{m-1}.
\eeq
Since $\Gr \le t \le 1$, we infer from \eqnref{Grasymp} that if $l \ge 1$
\beq\label{tGr2l}
1-t\Gr^{2l} \sim 1-\Gr^{2l+1} \sim (l+1) \sqrt{\Ge}.
\eeq
If $l=0$, we have
\beq\label{tGr2l2}
1-t\Gr^{2l} \lesssim \sqrt{\Ge}.
\eeq
Since $\Gb \approx \sqrt{\Ge}$ (see \eqnref{Gbasym}), it follows from \eqnref{4210} that
\beq\label{Gest}
|\Phi_{l,t}^{(m)}(z)| \lesssim \Gr^{2l} (l+1)^{m}
\eeq
for $m=0,1,2, \ldots$.

The desired estimate \eqnref{vjlest} for $j=1$ follows from \eqnref{Gest}. In fact,
we have
$$
v_{1}^{(n)}(z)= V_1^{(n)}(\Phi_{l,t}(z)) \Phi_{l,t}'(z)^n + \sum_{k=1}^{n-1} V_1^{(k)}(\Phi_{l,t}(z)) \sum_{\sum_{s=1}^t k_s m_s=n-k} C_{k,s} \Phi_{l,t}^{(k_s)}(z)^{m_s}
$$
for some constant $C_{k,s}$. It then follows from \eqnref{Gest} that
\begin{align*}
\| v_{1}^{(n)} \|_{C^{0}(\ol{D_1})} \lesssim \| V_1 \|_{C^{n}(\ol{D_1})} \left( \Gr^{2ln} (l+1)^{n} + \sum_{k=1}^{n-1} \sum_{\sum_{s=1}^t k_s m_s=n-k} \Gr^{2lm_s} (l+1)^{k_s m_s} \right),
\end{align*}
which leads us to \eqnref{vjlest} for $j=1$.

To prove \eqnref{vjlest} for $j=2$, we see as before that
\beq
\Psi_{l,t}^{(m)}(z) = \frac{m! \Gb^2 t\Gr^{-2l-1} (1- t\Gr^{-2l-1})^{m-1}}{(\Gb t\Gr^{-2l-1} - (1-t\Gr^{-2l-1}) z)^{m+1}}
\eeq
for $m=1,2, \ldots$, which can be written in terms of the $\Gz=T(z)$ variable as
$$
\Psi_{l,t}^{(m)}(z) = \Gb^{1-m} t\Gr^{-2l-1} (1-t\Gr^{-2l-1})^{m-1} \frac{(\Gz-1)^{m+1}}{(t\Gr^{-2l-1} \Gz - 1)^{m+1}}.
$$
Since $|\Gz| \ge R_2 >1$ and $1 \ge t\Gr^{-2l-1}$, we have
$$
|\Gz-1| \le 2 |\Gz - t^{-1}\Gr^{2l+1} |,
$$
and hence
$$
\left| \frac{(\Gz-1)^{m+1}}{(t\Gr^{-2l-1} \Gz - 1)^{m+1}} \right| \le \Gr^{2l} 2^{m+1}.
$$
As before, we have
\beq
t\Gr^{-2l-1} -1  \sim (l+1)\Gb.
\eeq
Thus we have
\beq\label{Psiest}
|\Psi_{l,t}^{(m)}(z)| \lesssim \Gr^{2l} (l+1)^{m}
\eeq
for $m=0,1,2, \ldots$.

The estimate \eqnref{vjlest} for $j=2$ can be proved similarly to the case when $j=1$ using \eqnref{Psiest}.
\end{proof}

\medskip

\noindent{\sl Proof of Theorem \ref{thm:main1}}.
We first note that $4\Gl_1\Gl_2 >0$. Let $\tilde{h}_j := h_j \circ T$ so that $\Re \tilde{h}_j=H_j$ (see \eqnref{hjdef}). We see from \eqnref{w1def2} that
\beq\label{w1circT}
(w_1 \circ T)(z) = \sum_{l=0}^\infty \frac{\tilde{h}_1 (T^{-1} \Gr^{2l} T(z)) }{(4\Gl_1\Gl_2)^{l+1}}.
\eeq
It follows from \eqnref{vjlest} that
$$
\| w_1 \circ T \|_{C^{n+1}(\ol{D_1})} \lesssim \sum_{l=0}^\infty
\frac{\Gr^{2l} (l+1)^{n+1}}{(4\Gl_1\Gl_2)^{l+1}} \| H_1 \|_{C^{n+1}(\ol{D_1})} .
$$
Since
$$
\sum_{l=0}^\infty \frac{\Gr^{2l} (l+1)^{n+1}}{(4\Gl_1\Gl_2)^{l+1}} \lesssim \frac{1}{(4\Gl_1\Gl_2 - \Gr)^{n+1}},
$$
we have
\beq\label{w1Tz}
\| \nabla (w_1 \circ T) \|_{C^{n, \Ga}(\ol{D_1})} \lesssim \frac{\| H_1 \|_{C^{n+1}(\ol{D_1})}}{(4\Gl_1\Gl_2 - \Gr)^{n+1}}.
\eeq
Likewise, we have
\beq\label{w2Tz}
\| \nabla (w_2 \circ T) \|_{C^{n, \Ga}(\ol{D_2})} \lesssim \frac{\| H_2 \|_{C^{n+1}(\ol{D_2})}}{(4\Gl_1\Gl_2 - \Gr)^{n+1}}.
\eeq

One can show that all the terms in the expressions \eqnref{A1def} and \eqnref{A2def} of the functions $A_1$ and $A_2$ can be estimated using \eqnref{w1Tz} and \eqnref{w2Tz}. For example, the term $w(R_1^2 \Gzbar^{-1})$ for $R_1 < |\Gz|$ (in \eqnref{A1def}) can be estimated as follows: Note that
$$
w_1(R_1^2 \Gzbar^{-1}) = w_1(R_1^2 \ol{T(z)}^{-1}) = (w_1 \circ T) (G(z)),
$$
where
\beq\label{Vone}
G(z) := T^{-1} (R_1^2 \ol{T(z)^{-1}}).
\eeq
Note that
$$
G(z)= \frac{\Gb(\Gb+z)}{(R_1^2-1) z - \Gb}.
$$
Thus, we have
$$
G^{(n)}(z)= -\frac{n! \Gb^2 R_1^2 (R_1^2-1)^{n-1}}{(\Gb -(R_1^2-1) z)^{n+1}}= - n!\Gb^{1-n} R_1^2 (R_1^2-1)^{n-1}  \frac{(\Gz-1)^{n+1}}{(\Gz - R_1^2)^{n+1}}.
$$
Since $R_1 \le |\Gz|$ and $R_1<1$, we have
$$
|\Gz-1| \lesssim |\Gz - R_1^2|.
$$
Since $R_1=1+O(\sqrt{\Ge})$ as shown in \eqnref{Rest}, we have
$$
|R_1^2-1| \lesssim \sqrt{\Ge}.
$$
Thus we have
\beq\label{Voneest}
|G^{(n)}(z)| \lesssim 1
\eeq
for all $z \in \Rbb^2 \setminus D_1$. Since $G$ maps $\Rbb^2 \setminus D_1$ into $D_1$, we have from \eqnref{w1Tz} and \eqnref{Voneest} that
\beq
\| (w_1 \circ T) \circ G \|_{C^{n+1}(\Rbb^2 \setminus D_1)} \lesssim \| w_1 \circ T \|_{C^{n+1}(\ol{D_1})} \lesssim \frac{\| H_1 \|_{C^{n+1}(\ol{D_1})}}{(4\Gl_1\Gl_2 - \Gr)^{n+1}}.
\eeq

So far we showed that the following estimate holds independently of $k_j$ and $\Ge$ for $j=1,2$:
$$
\| \nabla A_j \|_{n,0} \lesssim \frac{\| H_j \|_{C^{n+1}(\ol{D_1})}}{(4\Gl_1\Gl_2 - \Gr)^{n+1}}.
$$
It follows from \eqnref{firstrep} that
\beq\label{6000}
\| \nabla (u-F) \|_{n,0} \lesssim \frac{\| H_1 \|_{C^{n+1}(\ol{D_1})} + \| H_2 \|_{C^{n+1}(\ol{D_2})}}{(4\Gl_1\Gl_2 - \Gr)^{n+1}} .
\eeq

If $f=\nabla \cdot g$ and $g \in C^{n, \Ga}(\Rbb^2)$ for some nonnegative integer $n$ and $0<\Ga<1$, then $F \in C^{n+1, \Ga}(\Rbb^2)$ and
\beq\label{CZest}
\| F\|_{n+1, \Ga} \lesssim \| g \|_{n,\Ga}^*
\eeq
by the Calder\'on-Zygmund estimates. By regularity estimates of the boundary value problem, we have $H_j \in C^{n+1, \Ga}(\ol{D_j})$ and
\beq\label{H1est}
\| H_j \|_{C^{n+1, \Ga}(\ol{D_j})} \lesssim \| F\|_{C^{n+1, \Ga}(\ol{D_j})} \lesssim \| g \|_{n,\Ga}^*.
\eeq
It thus follows from \eqnref{6000} that
\beq\label{nablauest}
\| u \|_{n+1,0} \lesssim \frac{\| g \|_{n,\Ga}^*}{(4\Gl_1\Gl_2 - \Gr)^{n+1}} .
\eeq
Since
\beq\label{Grest}
\Gr \sim 1-r_* \sqrt{\Ge}
\eeq
by \eqnref{Grasymp}, the desired estimate \eqnref{mainest1} follows from \eqnref{nablauest}.

A proof of optimality of \eqnref{mainest1} will be given in Section \ref{sec:opti}. This completes the proof of Theorem \ref{thm:main1}. \qed

\bigskip

\noindent{\sl Proof of Theorem \ref{thm:main3}}.
If $(k_1-1)(k_2-1) <0$, then $\Gl_1\Gl_2 <0$. Let $\Gl:= -4\Gl_1\Gl_2$ and $Y_l^t(z):= \tilde{h}_1 (T^{-1} (t\Gr^{2l} T(z))) = (\tilde{h}_1 \circ \Phi_{l,t})(z) $ ($\tilde{h}_j := h_j \circ T$ as before). We can rewrite \eqnref{w1circT} as
\begin{align*}
(w_1 \circ T)(z) &= -\frac{1}{\Gl} \sum_{l=0}^\infty \frac{(-1)^{l} Y_l^1(z)}{\Gl^{l}} \nonumber \\
&= -\frac{1}{\Gl}\sum_{i=0}^\infty \left( \frac{1}{\Gl^{2i}} - \frac{1}{\Gl^{2i+1}} \right) Y_{2i}^1(z) +  \frac{1}{\Gl^{2i+1}} (Y_{2i}^1(z)-Y_{2i+1}^1(z)) \\
&=: -\frac{1}{\Gl}\sum_{i=0}^\infty (I_i(z) + J_i(z)). \label{w1circT2}
\end{align*}

Since $\Gl >1$, we have
$$
\left| \frac{1}{\Gl^{2i}} - \frac{1}{\Gl^{2i+1}} \right| = \frac{\Gl-1}{\Gl^{2i+1}} \le \frac{\Gl-1}{\Gl^{i}}.
$$
It thus follows from \eqnref{vjlest} that
$$
\| I_i \|_{C^{n+1}(\ol{D_1})} \lesssim (\Gl-1) \frac{\Gr^{i} (i+1)^{n+1}}{\Gl^{i}} \| H_1 \|_{C^{n+1}(\ol{D_1})}.
$$
Thus we have
\beq
\left\| \sum_{i=0}^\infty I_i \right\|_{C^{n+1}(\ol{D_1})} \lesssim
\frac{\Gl-1}{(\Gl-\Gr)^{n+1}} \| H_1 \|_{C^{n+1}(\ol{D_1})},
\eeq
which together with \eqnref{H1est} yields
\beq\label{Isum}
\left\| \sum_{i=0}^\infty I_i \right\|_{C^{n+1}(\ol{D_1})} \lesssim
\frac{\| g \|_{n,\Ga}^*}{(\Gl-\Gr)^{n}} .
\eeq

On the other hand, we have
\begin{align}
Y_{2i}^1(z)-Y_{2i+1}^1(z) &= Y_{2i}^1(z)-Y_{2i}^{\Gr^2}(z) \nonumber \\
&= \int_{\Gr^2}^1 \frac{\p}{\p t} Y_{2i}^t(z)dt =
\int_{\Gr^2}^1 h_1' (\Phi_{2i,t}(z)) \frac{\p}{\p t} \Phi_{2i,t}(z) dt . \label{4220}
\end{align}
We see from \eqnref{Phideri2} that
\begin{align*}
\frac{\p}{\p t} \Phi_{2i,t}^{(m)}(z) &= \frac{\Gb^{1-m} \Gr^{4i} (\Gz-1)^{m+1}}{(t\Gr^{4i} \Gz - 1)^{m+1}}  \\
& \quad \times  \left[ (1-t\Gr^{4i})^{m-1} - (m-1) t\Gr^{4i}(1-t\Gr^{4i})^{m-2} - \frac{(m+1)\Gr^{4i} \Gz}{t\Gr^{4i} \Gz - 1} \right] \\
& =: M_1 + M_2 + M_3.
\end{align*}
Using \eqnref{4209} and \eqnref{tGr2l}, we see that
$$
|M_1| \lesssim \Gr^{i} (i+1)^{m}, \quad
|M_2| \lesssim \frac{\Gr^{i} (i+1)^{m-1}}{\sqrt{\Ge}}.
$$
Since
$$
|t\Gr^{4i} \Gz -1| \ge 1- t\Gr^{i} R_1
$$
for any $\Gz$ with $|\Gz| \le R_1$, we have
$$
|t\Gr^{4i} \Gz -1| \gtrsim \sqrt{\Ge} (i+1).
$$
Thus we have
$$
|M_3| \lesssim \frac{\Gr^{i} (i+1)^{m-1}}{\sqrt{\Ge}}.
$$
It then follows that
\beq\label{4221}
\sum_{z \in D_1} \left| \frac{\p}{\p t} \Phi_{2i,t}^{(m)}(z) \right| \lesssim \Gr^{i} (i+1)^{m} + \frac{\Gr^{i} (i+1)^{m-1}}{\sqrt{\Ge}}.
\eeq

By \eqnref{vjlest}, we have
$$
\| g \circ \Phi_{2i,t} \|_{C^{n+1}(\ol{D_1})} \lesssim \Gr^{i} i^{n+1} \| g \|_{C^{n+1}(\ol{D_1})}
\lesssim \Gr^{i} i^{n+1} \| g \|_{n+1,\Ga}^*.
$$
We then infer from \eqnref{4220} and \eqnref{4221} that
$$
\| Y_{2i}^1-Y_{2i+1}^1 \|_{C^{n+1}(\ol{D_1})} \lesssim (1-\Gr^2) \left( \Gr^{i} (i+1)^{m} + \frac{\Gr^{i} (i+1)^{m-1}}{\sqrt{\Ge}} \right) \| g \|_{n+1,\Ga}^*.
$$
Since $1-\Gr^2 \lesssim \sqrt{\Ge}$, we have
\beq
\| Y_{2i}^1-Y_{2i+1}^1 \|_{C^{n+1}(\ol{D_1})} \lesssim \Gr^{i} \left( \sqrt{\Ge} (i+1)^{n+1} + (i+1)^{n} \right) \| g \|_{n+1,\Ga}^*.
\eeq
It then follows that
$$
\left\| \sum_{i=0}^\infty J_i \right\|_{C^{n+1}(\ol{D_1})} \lesssim \sum_{i=0}^\infty \frac{\Gr^{i}}{\Gl^i} \left( \sqrt{\Ge} (i+1)^{n+1} + (i+1)^{n} \right) \| g \|_{n+1,\Ga}^* .
$$
Since
\begin{align*}
\sum_{i=0}^\infty \frac{\Gr^{i}}{\Gl^i} \left( \sqrt{\Ge} (i+1)^{n+1} + (i+1)^{n} \right)
&\lesssim \frac{\sqrt{\Ge}}{(\Gl-\Gr)^{n+1}} + \frac{1}{(\Gl-\Gr)^{n}} \lesssim \frac{1}{(\Gl-\Gr)^{n}},
\end{align*}
we have
\beq
\left\| \sum_{i=0}^\infty J_i \right\|_{C^{n+1}(\ol{D_1})} \lesssim \frac{\| g \|_{n+1,\Ga}^*}{(\Gl-\Gr)^{n}}.
\eeq
This together with \eqnref{Isum} yields
\beq\label{4230}
\left\| w_1 \circ T \right\|_{C^{n+1}(\ol{D_1})} \lesssim \frac{\| g \|_{n+1,\Ga}^*}{(\Gl-\Gr)^{n}}.
\eeq

Similarly, one can show that
\beq\label{4231}
\left\| w_2 \circ T \right\|_{C^{n+1}(\ol{D_2})} \lesssim \frac{\| g \|_{n+1,\Ga}^*}{(\Gl-\Gr)^{n}}.
\eeq
Then, as before, one can show
$$
\| A_j \|_{n+1,0} \lesssim \frac{\| g \|_{n+1,\Ga}^*}{(\Gl-\Gr)^{n}}
$$
for $j=1,2$, which yields the desired estimate \eqnref{mainest3-1}.
This completes the proof. \qed

\bigskip

Since the solution to \eqnref{HCE} admits the same representation as the one in Proposition \ref{prop:rep}, Theorem \ref{thm:main4} and \ref{thm:main5} can be proved in the same way as Theorem \ref{thm:main1} and \ref{thm:main3}.

\section{Optimality}\label{sec:opti}

In this section we show that there is $f$ (or $H$) such that the reverse inequality holds in \eqnref{mainest1}, \eqnref{mainest4-1}, \eqnref{mainest3-1}, and \eqnref{mainest4-2}.

Let $F$ be a smooth function in $\Rbb^2$ with a compact support such that $F(z)=x_1$ in a neighborhood of $\ol{D_1 \cup D_2}$. Let $f:= \GD F$. Then the relation \eqnref{Fdef} holds, namely, $F$ is the Newtonian potential of $f$. Note that $f=0$ in $D_1 \cup D_2$. The solution $H_j$ of \eqnref{bvp} is given by $H_j(z)=x_1$ in $D_j$ for $j=1,2$.
We show the following:
\begin{itemize}
\item[(i)] For the case when $(k_1-1)(k_2-1) >0$, we take $k_1=k_2=\infty$ and show that the solution $u$ to \eqnref{NCE} with $f$ defined above satisfies
\beq\label{reverse1}
|\nabla u(z)| \gtrsim \Ge^{-1/2}
\eeq
for some $z \in \Rbb^2 \setminus \ol{D}$.
\item[(ii)] For the case when $(k_1-1)(k_2-1) < 0$, we take either $k_1=0, \ k_2=\infty$ or $k_1=\infty, \ k_2=0$ and show that the solution $u$ to \eqnref{NCE} with $f$ defined above satisfies
\beq\label{reverse2}
|\nabla^2 u(z)| \gtrsim \Ge^{-1/2}
\eeq
for some $z \in \Rbb^2 \setminus \ol{D}$. ($\nabla u$ is bounded.)
\end{itemize}
Similar estimates hold for the solution to the homogeneous problem with $H(x)=x_1$.

According to the definition \eqnref{hjdef} of $h_j$, we have
\beq\label{hjexample}
h_j(\Gz)= \frac{\Gb}{\Gz-1}, \quad \Gz \in D_j^*.
\eeq

Suppose that $k_1=k_2=\infty$. Then, $\Gl_1+\Gl_2=1$ and $\Gl_1-\Gl_2=0$.
Since $4 \Gl_1\Gl_2=1$ in this case, it follows from the formulas \eqnref{A1def} and \eqnref{A2def} of $A_j$ that
$$
V(\Gz) = A_1(\Gz) + A_2(\Gz) = w_1(R_1^2 \Gzbar^{-1}) - w_1(\Gr^2 \Gz) + w_2(R_2^2 \Gz^{-1}) - w_2(\Gr^{-2} \Gzbar)
$$
if $R_1 < |\Gz| \leq R_2$. Then the formulas \eqnref{w1def2} and \eqnref{w2def2} for $w_j$ yield
\begin{align*}
\p_\Gz V(\Gz) &= - \Gr^2 w_1'(\Gr^2 \Gz) - R_2^2 \Gz^{-2} w_2'(R_2^2 \Gz^{-1}) \\
& = -\sum_{l=0}^\infty \Gr^{2l+2} h_1' (\Gr^{2l+2}\Gz) - \sum_{l=0}^\infty \Gr^{-2l}R_2^2 \Gz^{-2} h_2' (\Gr^{-2l}R_2^2 \Gz^{-1}) .
\end{align*}
It thus follows from \eqnref{hjexample} that
\beq\label{pzV}
\p_\Gz V(\Gz) = \Gb \sum_{l=0}^\infty \left[ \frac{\Gr^{2l+2}}{(1-\Gr^{2l+2}\Gz)^2} + \frac{\Gr^{2l} R_2^{-2}}{(1-\Gr^{2l} R_2^{-2}\Gz)^2} \right] .
\eeq
In particular, we have
$$
\p_\Gz V(-1) = \Gb \sum_{l=0}^\infty \left[ \frac{\Gr^{2l+2}}{(1+\Gr^{2l+2})^2} + \frac{\Gr^{2l} R_2^{-2}}{(1+\Gr^{2l} R_2^{-2})^2} \right] \gtrsim \frac{1}{1-\Gr} \gtrsim \frac{1}{\Gb}.
$$
Note that
\beq\label{dGzdz}
\frac{\p^k \Gz}{\p z^k}= (-1)^k \frac{k! \Gb}{z^{k+1}} = (-1)^k k! \frac{(\Gz-1)^{k+1}}{\Gb^{k}}, \quad k=1,2,\ldots.
\eeq
Thus we have
\beq
\big| \p_z (V\circ T)(T^{-1}(-1)) \big| = \Big| \p_\Gz V(-1) \frac{\p \Gz}{\p z} \Big|_{\Gz=-1} \Big| \gtrsim \frac{1}{\Gb},
\eeq
and we arrive at \eqnref{reverse1}.

\medskip

We now consider the case when $(k_1-1)(k_2-1) < 0$. If $k_1=0$ and $k_2=\infty$, then $\Gl_1=-1/2$ and $\Gl_2=1/2$, and hence $\Gl_1+\Gl_2=0$, $\Gl_1-\Gl_2=-1$, and $4 \Gl_1\Gl_2=-1$. We have from \eqnref{A1def} and \eqnref{A2def} that
$$
V(\Gz) = -w_1(\Gr\Gz) - w_1(\Gr^2 \Gz) + w_2(\Gr^{-1} \Gzbar) - w_2(\Gr^{-2} \Gzbar)
$$
if $R_1 < |\Gz| \leq R_2$. Thus we have
\beq
\p_\Gz V(\Gz) = - \Gr w_1'(\Gr\Gz) - \Gr^2 w_1'(\Gr^2 \Gz).
\eeq

We see from \eqnref{w1def2} and \eqnref{hjexample} that
$$
w_1(\Gz) = \Gb \sum_{l=0}^\infty (-1)^{l+1} \frac{1}{1-\Gr^{2l}\Gz} = -\Gb \sum_{k=0}^\infty \frac{\Gz^k}{1+\Gr^{2k}},
$$
and hence
$$
w_1'(\Gz) = -\Gb \sum_{k=1}^\infty \frac{k\Gz^{k-1}}{1+\Gr^{2k}}.
$$
Thus, we have
$$
\p_\Gz V(\Gz) = \Gb \sum_{k=1}^\infty \frac{k (\Gr^{k} + \Gr^{2k}) \Gz^{k-1}}{1+\Gr^{2k}} .
$$
We also have
$$
\p_\Gz^2 V(\Gz) = \Gb \sum_{k=2}^\infty \frac{k(k-1) (\Gr^{k} + \Gr^{2k}) \Gz^{k-2}}{1+\Gr^{2k}} .
$$
Note that
$$
\p_z^2 (V\circ T)(z) = \p_\Gz^2 V(\Gz) \left(\frac{\p \Gz}{\p z} \right)^2 + \p_\Gz V(\Gz) \frac{\p^2 \Gz}{\p z^2}.
$$
Thus \eqnref{dGzdz} yields
$$
\p_z^2 (V\circ T)(z) = \frac{(\Gz-1)^3}{\Gb} \Big[ \sum_{k=1}^\infty \frac{k(k+1) (\Gr^{k} + \Gr^{2k}) \Gz^{k-1}}{1+\Gr^{2k}} +  \sum_{k=2}^\infty \frac{k(k-1) (\Gr^{k} + \Gr^{2k}) \Gz^{k-2}}{1+\Gr^{2k}} \Big].
$$
Note that
$$
\int_{-\pi}^{\pi} \Gb \p_z^2 (V\circ T)(T^{-1}(e^{i\Gt})) d\Gt= - 2\pi C,
$$
where
$$
C= \frac{2(\Gr + \Gr^{2})}{1+\Gr^{2}} +  \frac{2(\Gr^{2} + \Gr^{4})}{1+\Gr^{4}}.
$$
It means that there is a point, say $\Gz_*$, such that $|\Gz_*|=1$ and
$$
\Gb \p_z^2 (V\circ T)(T^{-1}(\Gz_*)) \le - C.
$$
Note that $C \ge 1$ if $\Ge$ is sufficiently small. Thus we have
\beq\label{7100}
\big| \p_z^2 (V\circ T)(T^{-1}(\Gz_*)) \big| \gtrsim \frac{1}{\Gb} ,
\eeq
which yields \eqnref{reverse2}.

If $k_1=\infty$ and $k_2=0$, then $\Gl_1+\Gl_2=0$, $\Gl_1-\Gl_2=1$, and $4 \Gl_1\Gl_2=-1$. We have from \eqnref{A1def} and \eqnref{A2def} that
$$
V(\Gz) = w_1(\Gr\Gz) - w_1(\Gr^2 \Gz) - w_2(\Gr^{-1} \Gzbar) - w_2(\Gr^{-2} \Gzbar)
$$
if $R_1 < |\Gz| \leq R_2$. Thus we have
\beq
\p_{\Gzbar} V(\Gz) = - \Gr^{-1} w_2'(\Gr^{-1} \Gzbar) - \Gr^{-2} w_2'(\Gr^{-2} \Gzbar).
\eeq
Now in the same way as above, one can show that there is a $\Gz_{**}$ such that $|\Gz_{**}|=1$ and
\beq\label{7200}
\big| \p_{\zbar}^2 (V\circ T)(T^{-1}(\Gz_{**})) \big| \gtrsim \frac{1}{\Gb} ,
\eeq
which again yields \eqnref{reverse2}.



\end{document}